\documentclass[12pt,leqno]{amsart}
\usepackage{amssymb,amsmath,amsthm,bm}
\usepackage{graphicx}
\usepackage{color}
 \usepackage[textsize=tiny]{todonotes}
\usepackage[normalem]{ulem}
\usepackage[bookmarksopen,bookmarksdepth=3,colorlinks,citecolor=red,pagebackref,hypertexnames=true]{hyperref}
\usepackage[msc-links,nobysame,non-sorted-cites,initials]{amsrefs}
\usepackage[inline]{enumitem}
\usepackage{mathtools}
\mathtoolsset{showonlyrefs}
\usepackage{tikz}
\usetikzlibrary{arrows.meta,calc}

\renewcommand{\eprint}[1]{\href{http://arxiv.org/abs/#1}{#1}}

\definecolor{darkblue}{RGB}{0,0,160}

\usepackage[lining,tabular]{libertine}
\usepackage{cabin}
\usepackage[libertine]{newtxmath}

\def\eps{\varepsilon} 
\def\d{{\mathrm d}}
\def\dist{\mathrm {dist}}
\def\dim{{\mathrm {dim}}}
\def\R {\mathbb{R}}
\def\N {\mathbb{N}}

\def\C {{\mathbb C}}

\def\M {{\mathrm M}}
\def\Z {{\mathbb Z}}

\def\ind{\mathbf{1}}

\newcommand{\Gr}{\operatorname{Gr}}

\newcommand{\pv}{\operatorname{p.v.}}

 \def\Xint#1{\mathchoice
     {\XXint\displaystyle\textstyle{#1}}%
     {\XXint\textstyle\scriptstyle{#1}}%
     {\XXint\scriptstyle\scriptscriptstyle{#1}}%
     {\XXint\scriptscriptstyle\scriptscriptstyle{#1}}%
     \!\int}
   \def\XXint#1#2#3{{\setbox0=\hbox{$#1{#2#3}{\int}$}
       \vcenter{\hbox{$#2#3$}}\kern-.5\wd0}}
   
   \def\avgint{\Xint-}
\makeatletter
\DeclareRobustCommand\widecheck[1]{{\mathpalette\@widecheck{#1}}}
\def\@widecheck#1#2{%
    \setbox\z@\hbox{\m@th$#1#2$}%
    \setbox\tw@\hbox{\m@th$#1%
       \widehat{%
          \vrule\@width\z@\@height\ht\z@
          \vrule\@height\z@\@width\wd\z@}$}%
    \dp\tw@-\ht\z@
    \@tempdima\ht\z@ \advance\@tempdima2\ht\tw@ \divide\@tempdima\thr@@
    \setbox\tw@\hbox{%
       \raise\@tempdima\hbox{\scalebox{1}[-1]{\lower\@tempdima\box
\tw@}}}%
    {\ooalign{\box\tw@ \cr \box\z@}}}
\makeatother

\newcounter{counter}

\numberwithin{equation}{section}
\numberwithin{counter2}{section}

\newtheorem{theorem}[counter]{Theorem}

\numberwithin{withinsub}{subsection}
\newtheorem{corollary}[subsection]{Corollary}
\newtheorem{lemma}[subsection]{Lemma} 
 
\newtheorem{proposition}[subsection]{Proposition}
\newtheorem{conjecture}[subsection]{Conjecture}
 
\theoremstyle{definition}
\newtheorem{definition}[subsection]{Definition}
 
\newtheorem*{remark*}{Remark}
\newtheorem*{warn*}{A word of warning}

\theoremstyle{definition}

 \theoremstyle{plain}

\numberwithin{figure}{section}

\begin{document}

\title[Directional singular integrals]{Maximal Directional Singular Integrals}

\author[F.\ Di Plinio]{Francesco Di Plinio} 
\address[F.\ Di Plinio]{Dipartimento di Matematica e Applicazioni ``R.\ Caccioppoli'' \\ \newline\indent Universit\`a di Napoli ``Federico II'', Via Cintia, Monte S.\ Angelo 80126 Napoli, Italy}
\email{\href{mailto:francesco.diplinio@unina.it}{\textnormal{francesco.diplinio@unina.it}}}

\author[I. Parissis]{Ioannis Parissis}
\address[I.\ Parissis]{Departamento de Matem\'aticas, Universidad del Pa\'is Vasco, Aptdo. 644, 48080 Bilbao, Spain and Ikerbasque, Basque Foundation for Science, Bilbao, Spain}

\email{\href{mailto:ioannis.parissis@ehu.eus}{\textnormal{ioannis.parissis@ehu.eus}}}
\thanks{I. Parissis is partially supported by grant PID2024-156267NB-I00 funded by MICIU/AEI/10.13039/501100011033 and cofunded by the European Union,
grant IT1875-26 of the Basque Government and IKERBASQUE.}


%%%%%%%%%%%%%%%%%%%%%%%%%%%%%% ABSTRACT ABSTRACT ABSTRACT
\begin{abstract}We survey maximal directional averages and directional singular integrals, and their connections with Kakeya--Nikodym phenomena, differentiation theory and the Carleson operator. The exposition is organized by the structure of the set of directions: lacunary and arbitrary finite sets, vector fields as in the conjectures of Zygmund and Stein, and, in higher dimensions, algebraic sets of directions and subspaces. The unifying tool is almost orthogonality, which rests on the ordering of the circle in the plane and on polynomial partitioning in higher dimensions. We explain where it yields sharp or nearly sharp bounds and where new ideas are needed, notably for multiscale averages along algebraic curves and for vector-valued single-annulus estimates, and we collect open problems.
\end{abstract}
%%%%%%%%%%%%%%%%%%%%%%%%%%%%%% ABSTRACT ABSTRACT ABSTRACT
\maketitle

%%%%%%%%%%%%%%%%%%%%%%%%%%%%%% SECTION SECTION SECTION
\section{Introduction} Our subject is maximal and singular directional operators. These act on functions of $\mathbb R^n$  by applying lower-dimensional averaging or singular integral operators along affine \(d\)-dimen\-sional subspaces of the Euclidean space, \(1\leq d<n\). Schematically, the orientation of these subspaces is allowed to vary in a prescribed set \(V\subset \Gr(d,n)\), where \(\Gr(d,n)\) denotes the Grassmannian of \(d\)-planes in \(\R^n\). The case \(d=1\) is the most classical, and also the most singular, corresponding to the more familiar directional operators along lines, with \(V\subset \mathbb S^{n-1}\). 

Directional operators are closely related to Kakeya--Nikodym type questions. The geometric problem corresponding to the \(d=1\) case asks how small a set in \(\R^n\) can be, in terms of a suitable notion of dimension, if it contains a unit line segment in every direction. Analytically, this connection is visible for uniformly distributed finite sets of directions: if \(V=V_\delta\subset \mathbb S^{n-1}\) is a \(\delta\)-net and the corresponding single-scale maximal averages satisfy \(L^n(\R^n)\)-bounds with losses \(O_\eps(\delta^{-\eps})\), for every \(\eps>0\), then one obtains the maximal Nikodym conjecture, and with it, the two being equivalent for lines, the maximal Kakeya conjecture. Replacing lines by \(d\)-dimensional affine subspaces leads to higher-dimensional analogues, parametrized by \(\Gr(d,n)\).

The focus is on how the mapping properties of these operators depend on the size, geometry, and possible additional structure of the set of admissible directions. In two dimensions the theory is quite well developed, and many of the main quantitative estimates are known in sharp or nearly sharp form, although important questions remain open for singular variants and for vector-field models. In higher dimensions, the picture is much less complete, with difficult open questions related to problems such as the maximal Kakeya conjecture. 

 Both in the plane and in higher dimensions, we organize the exposition according to the structure of the set of directions, beginning with lacunary sets and moving toward less restrictive configurations. This gives a genuine hierarchy of difficulty for maximal averages, while for singular integrals the picture is somewhat less ordered; see \S\ref{sec:planar}.

%%%%%%%%%%%%%%%%%%%%%%%%%%%%%% SECTION SECTION SECTION
 \subsection*{Roadmap} \S\ref{sec:cex} discusses the main obstructions, \S\ref{sec:almost-orthogonality} almost orthogonality, \S\ref{sec:planar} the planar theory, and \S\ref{sec:higher_dimensions}--\S\ref{sec:algebraic-AO} the higher-dimensional and algebraic theory. Open problems are collected in \S\ref{sec:open}.

%%%%%%%%%%%%%%%%%%%%%%%%%%%%%% SUBSECTION SUBSECTION SUBSECTION
\subsection{Directional averages} Let \(V\subset \mathbb S^{n-1}\) be a set of directions and let \(r>0\). We denote by \(\M_{v,r}\) the rough average in the direction \(v\in V\),
\begin{equation}
 \M_{v,r}f(x)\coloneqq \frac1{2r}\int_{-r}^r |f(x-tv)|\,\d t, \qquad x\in\R^n.
\end{equation}
The corresponding maximal averaging operators are
\begin{equation}
 \M_{V,r}f(x)\coloneqq \sup_{v\in V}\M_{v,r}f(x), \qquad \M_Vf(x)\coloneqq \sup_{v\in V,\,r>0}\M_{v,r}f(x),\qquad x\in\R^n.
\end{equation}
By scaling, the \(L^p\)-norm of the \emph{single-scale} operator \(\M_{V,r}\) is independent of \(r>0\), so it suffices to consider \(\M_{V,1}\). This operator is closely tied to Nikodym-type maximal estimates. To make this precise, let \(T_{\delta,v}(x)\) denote a tube of length \(1\), width \(\delta\), centered at \(x\), and pointing in the direction \(v\). The Nikodym maximal function at scale \(\delta\) is
\[
\mathcal N_\delta f(x)\coloneqq \sup_{v\in \mathbb S^{n-1}} \frac1{|T_{\delta,v}(x)|}\int_{T_{\delta,v}(x)} |f(y)|\,\d y ,\qquad x\in\R^n.
\]
If \(V_\delta\subset \mathbb S^{n-1}\) is a \(\delta\)-net, then the tube average in an arbitrary direction is controlled, up to harmless dimensional constants, by a line average in a nearby direction after thickening the function at scale \(\delta\). More precisely,
\begin{equation}\label{eq:nikodym-comparison}
 \left\|\mathcal N_\delta f\right\|_{L^p(\R^n)}
 \lesssim
 \left\|\M_{V_\delta,1} (|f|*\psi_\delta)\right\|_{L^p(\R^n)},
\end{equation}
where \(\psi_\delta=\delta^{-n}\psi(\cdot/\delta)\) for a suitable nonnegative smooth bump function \(\psi\). Thus \(\M_{V_\delta,1}\), averaging over segments rather than tubes, is a singular model for the Nikodym maximal problem. The multiscale operator \(\M_V\) is naturally connected with differentiation theory.

Fourier-analytic arguments conveniently replace the rough averages by smooth ones with compactly supported Fourier multiplier. Fix once and for all a nonnegative \(\varphi\in\mathcal S(\R)\), bounded from below on \([-1,1]\) and with \(\widehat\varphi\) supported in \([-1/2,1/2]\); the construction of such a \(\varphi\) is standard, see for instance \cite{PR}*{\S1}. Writing \(\varphi_r(t)\coloneqq r^{-1}\varphi(t/r)\), put
\[
 A_{v,r}f(x)\coloneqq \int_\R f(x-tv)\varphi_r(t)\,\d t,\hspace{.5em} x\in\R^n, \qquad A_{V,r}f\coloneqq \sup_{v\in V}\left|A_{v,r}f\right|, \qquad A_Vf\coloneqq \sup_{v\in V,\,r>0}\left|A_{v,r}f\right|.
\]
The lower bound on \(\varphi\) gives \(\M_{v,r}f\lesssim A_{v,r}f\) for \(f\geq0\), while its rapid decay gives \(|A_{v,r}f|\lesssim\sum_{k\geq0}2^{-10k}\M_{v,2^kr}f\), on splitting the \(t\)-integral into dyadic annuli. Taking suprema in \(v\), and in \(r\) where needed,
\begin{equation}\label{eq:smooth-rough}
\left\|\M_{V,r}\right\|_{L^p(\R^n)} \simeq \left\|A_{V,r}\right\|_{L^p(\R^n)},
\qquad \left\|\M_V\right\|_{L^p(\R^n)} \simeq \left\|A_V\right\|_{L^p(\R^n)}, \qquad 0<p\leq\infty,
\end{equation}
with constants independent of \(V\) and \(r\). We pass freely between the two models, the advantage of the smooth one being its simple Fourier description
\begin{equation}
\widehat{A_{v,r}f}(\xi)=\widehat\varphi(r\,\xi\cdot v)\widehat f(\xi), \qquad \xi\in\R^n,
\end{equation}
which exhibits \(A_{v,r}\) as a multiplier adapted to the band \(|\xi\cdot v|\lesssim r^{-1}\), that is, to an \(r^{-1}\)-neighborhood of the hyperplane \(v^\perp\). This frequency localization is one of the starting points for the conical decompositions and almost-orthogonality arguments discussed below.

%%%%%%%%%%%%%%%%%%%%%%%%%%%%%% SUBSECTION SUBSECTION SUBSECTION
\subsection{Directional singular integrals and multipliers} The singular integral counterparts of the directional averages above are defined as follows. The basic model is the directional Hilbert transform. For \(v\in \mathbb S^{n-1}\), define
\begin{equation}
 H_vf(x)\coloneqq \frac1\pi\operatorname{p.v.} \int_\R f(x-tv)\frac{\d t}{t},\qquad x\in\R^n.
\end{equation}
Given a set of directions \(V\subset \mathbb S^{n-1}\), the associated maximal directional Hilbert transform is
\begin{equation}
 H_Vf(x)\coloneqq \sup_{v\in V}|H_vf(x)|,\quad x\in\R^n, \qquad \widehat{H_vf}(\xi)=-i \operatorname{sgn}(\xi\cdot v)\widehat f(\xi), \quad \xi\in\R^n.
\end{equation}
By taking linear combinations with the identity operator, the signum may be replaced by half-space projection with multiplier \(\ind_{[0,\infty)}(\xi\cdot v)\). More generally, one can replace the multiplier of the Hilbert transform  by a one-dimensional multiplier $m$ applied in the \(v\)-direction. If \(m:\R\setminus\{0\}\to\C\), define
\begin{equation}\label{eq:max-dir-mult}
 T_{m,v}f(x)\coloneqq \int_{\R^n} m(\xi\cdot v)\widehat f(\xi)e^{2\pi i x\cdot \xi}\,\d\xi,\qquad T_{m,V}f(x)\coloneqq \sup_{v\in V}|T_{m,v}f(x)|,\qquad x\in\R^n.
\end{equation}
Typically we assume that \(m\in \mathrm{HM}(\R)\) is a one-dimensional \emph{H\"ormander--Mihlin} multiplier, that is, for a sufficiently large integer \(M\),
\begin{equation}
\|m\|_{\mathrm{HM}(\R)}\coloneqq \sup_{0\leq j\leq M}\sup_{\eta\neq0} |\eta|^j |m^{(j)}(\eta)|<\infty.
\end{equation}
There is also a vector-valued version of the same model. Given a family of directions \(V\), one may consider square functions of the form
\begin{equation}
 S_V\vec f(x)\coloneqq
 \left(\sum_{v\in V}|T_{m,v}f_v (x)|^2\right)^{1/2},\qquad \vec f=\{f_v\}_{v\in V},\qquad x\in\R^n,
\end{equation}
or variants in which the multipliers themselves are indexed by the set of directions \(V\). These directional square functions should be viewed as \(\ell^2\)-valued directional singular integrals. They are useful, for instance, in the study of polygon multipliers and of other multipliers whose geometry is associated with a set of directions $V$.

%%%%%%%%%%%%%%%%%%%%%%%%%%%%%% SECTION SECTION SECTION
\section{Kakeya counterexamples and other obstructions}\label{sec:cex} We mention several different sources of quantitative lower bounds for maximal directional operators. The most elementary one comes from radial counterexamples, and is best framed for a \(\delta\)-net \(V_\delta\subset\mathbb S^{n-1}\), the numerology being \(\#V_\delta\simeq\delta^{1-n}\).

Let \(0<\delta\ll1\), let \(x\in B(0,1)\), and choose \(v\in V_\delta\) with \(\left|v-x/|x|\right|<\delta\). Then \(|x-|x|v|\leq\delta\), and hence \(|x-tv|\leq2\delta\) whenever \(\bigl||x|-t\bigr|\leq\delta\). Testing on \(f_\delta\coloneqq\ind_{B(0,2\delta)}\), we obtain $\M_{V_\delta,1}f_\delta(x)\gtrsim \delta$ on $ B(0,1)$. Since \(\|f_\delta\|_p^p\simeq\delta^n\), the weak-type calculation gives
\begin{equation}\label{eq:radial-subcritical}
\|\M_{V_\delta,1}\|_{L^p(\R^n)\to L^{p,\infty}(\R^n)} \gtrsim\delta^{1-n/p},\qquad 1\leq p<n.
\end{equation}
At the endpoint \(p=n\), take \(g_\delta(x)\coloneqq|x|^{-1}\ind_{\{\delta<|x|<1\}}(x)\). Then \(\|g_\delta\|_n^n\simeq\log(\delta^{-1})\), while the same radial segments as above give \(\M_{V_\delta,1}g_\delta\gtrsim\log(\delta^{-1})\) on a fixed annulus in \(B(0,1)\). Repeating the weak-type calculation,
\begin{equation}\label{eq:radial-single-scale-endpoint}
\|\M_{V_\delta,1}\|_{L^n(\R^n)\to L^n(\R^n)} \geq \|\M_{V_\delta,1}\|_{L^n(\R^n)\to L^{n,\infty}(\R^n)} \gtrsim\left(\log\tfrac1\delta\right)^{(n-1)/n}.
\end{equation}
For the multiscale operator, the same example gives \(\M_{V_\delta}g_\delta(x)\gtrsim |x|^{-1}\log(|x|/\delta)\) for \(2\delta<|x|<1\), and hence
\begin{equation}\label{eq:radial-multiscale-endpoint}
\|\M_{V_\delta}\|_{L^n(\R^n)\to L^n(\R^n)}\gtrsim\log(\delta^{-1}).
\end{equation}
Letting \(\delta\to0\) in \eqref{eq:radial-subcritical} shows in particular that \(\M_{\mathbb S^{n-1},1}\), and hence also \(\M_{\mathbb S^{n-1}}\), is
unbounded on \(L^p(\R^n)\) for \(1\leq p<n\). We will use the case \(n=2\) of
\eqref{eq:radial-multiscale-endpoint} in \S\ref{sec:AO-plane}. These examples are sharp within the radial class; see \cites{CaHeSo,CaHeSoer} and also \cite{DuoVargas} for related planar results for direction sets of prescribed Minkowski dimension. Thus radial counterexamples prevent boundedness on \(L^p(\R^n)\) for \(p\leq n\). It is at this threshold, \(p=n\) in \(\R^n\), that Kakeya--Nikodym phenomena become particularly relevant. For \(\delta^{-\eps}\)-free bounds the right invariant is not the Minkowski dimension but the quasi-Assouad dimension: for the planar Nikodym operator restricted to \(\Theta\subset\mathbb S^1\) with
\(\dim_{\mathrm{qA}}\Theta=s\geq\tfrac12\), Orponen--Roy \cite{OrpRoy} show that the critical exponent is exactly \(1+s\).

The second, geometric, source of lower bounds is  \emph{Kakeya compression}. For
\(0<\delta\ll1\), a family of \(\delta\times1\) tubes \(\{T_v:v\in V_\delta\}\) pointing in \(\delta\)-separated directions can have union much smaller than the
union of \(O(1)\)-translates \(\{\widetilde T_v\}\) of the same tubes along their long directions, 
\[
\left|\bigcup_{v\in V_\delta} T_v \right| \leq C(\delta) \left|\bigcup_{v\in V_\delta} \widetilde T_v \right|, \qquad C(\delta)\to0\quad\text{as}\quad \delta\to0^+.
\]
Any such configuration is a counterexample, in the sense that  \(\M_{V_\delta,1}\ind_{\cup_v T_v} \gtrsim1\) on \(\cup_v\widetilde T_v\), since the $O(1)$ segment through a point of \(\widetilde T_v\) in the direction \(v\) traverses \(T_v\), so that \(\|\M_{V_\delta,1}\|_{L^p}\gtrsim C(\delta)^{-1/p}\). In the plane the construction of Besicovitch \cite{Besi}, later optimized by Schoenberg \cite{Schoen}, attains \(C(\delta)\simeq(\log(\delta^{-1}))^{-1}\) with the \(\widetilde T_v\) disjoint; see \cite{PeBi}*{Lemma 9.3.2}. Hence
\begin{equation}
 \|\M_{V_\delta,1}\|_{L^p(\R^2)} \gtrsim \bigl(\log(1/\delta)\bigr)^{1/p}, \qquad 2\leq p<\infty.
\end{equation}
Radial and Kakeya-type examples apply equally to maximal directional singular integrals, which however carry a further obstruction, unrelated to compression and rooted in lower bounds for rearrangements of Fourier series. Karagulyan \cite{Kara} proved that in the plane the maximal directional Hilbert transform satisfies \(L^2\)-lower bounds of order \(\sqrt{\log \#V}\) for arbitrary finite sets of directions, and {\L}aba--Marinelli--Pramanik \cite{LMP} extended this to higher dimensions and to all \(p\in(1,\infty)\). Boundedness may therefore fail
even where Kakeya counterexamples are absent. The phenomenon is not exclusive to
directional operators: maximally dilated H\"ormander--Mihlin multipliers are
typically unbounded, even on the real line; see \cites{CGHS,GHS}.

%%%%%%%%%%%%%%%%%%%%%%%%%%%%%% SECTION SECTION SECTION
\section{Almost orthogonality as a localization principle} \label{sec:almost-orthogonality} A recurring principle in the study of directional operators, often tied to a low-high splitting, is that a set of directions can be localized into geometrically separated pieces, while the interaction between these pieces is encoded by a new set of directions of lower geometric complexity. We will refer rather loosely to estimates of this form as \emph{almost-orthogonality principles}.

For $n\geq 2$, let $V\subset\mathbb S^{n-1}$ be a discrete null set of directions, finite or countable. Let \(\Omega\subset\mathbb S^{n-1}\) be an \((n-2)\)-dimensional set, and let \(\{\Omega_j\}\) be the connected components of \(\mathbb S^{n-1}\setminus\Omega\). Associated with $\Omega$, one seeks $V_\Omega\subset\Omega$, with $\#V_\Omega\leq \#V$, for which an estimate of the form
\begin{equation}\label{eq:AO}
\|\M_V f\|_{L^2(\R^n)} \leq C_1\|\M_{V_\Omega} f\|_{L^2(\R^n)}+C_2\sup_j \|\M_{V\cap\Omega_j}\|_{L^2(\R^n)}\|f\|_{L^2(\R^n)}
\end{equation}
holds. One could postulate similar principles in other $L^p$ or even weak-$L^p$ spaces.

Estimate \eqref{eq:AO} reduces the problem to the smaller or simpler pieces \(V\cap\Omega_j\) and a residual set \(V_\Omega\) of lower dimension. For iteration, having \(C_1=1\) or \(C_2=1\) can be decisive, depending on which term carries the recursion. In the plane, the separating set is zero-dimensional and Theorem~\ref{thm:AO-plane} gives \(C_2=1\), yielding the sharp cardinality bound in \S\ref{sec:AO-plane}. In the algebraic version of \S\ref{sec:algebraic-AO}, obtained by polynomial partitioning, the second coefficient depends on the degree of the partitioning polynomial.

%%%%%%%%%%%%%%%%%%%%%%%%%%%%%% SECTION SECTION SECTION
\section{Directional operators in the plane: what is known and what remains open}\label{sec:planar} We pass from lacunary to uniformly distributed and arbitrary finite direction sets, then to cardinality-free bounds along vector fields. Before the vector-field discussion, we prove the almost-orthogonality principle underlying the arbitrary case. Higher-dimensional analogues are addressed in \S\ref{sec:higher_dimensions}.

The one result escaping that ladder is the maximal directional Hilbert transform on $L^2(\mathbb R^2)$. For a finite set \(V=\{v_1,\ldots,v_N\}\subset\mathbb S^1\)there holds
\begin{equation}\label{eq:HV-RM}
\left\|H_V\right\|_{L^2(\R^2)}\lesssim\log N
\end{equation}
with no hypothesis of any kind on \(V\). The matching lower bound is already attained on the equispaced set \(\{\mathrm e^{2\pi ik/N}:\,k=1,\ldots,N\}\), as shown in \cite{ciprianpams}*{Proposition 3.3}. We recall the folklore proof of \eqref{eq:HV-RM}. Identify antipodal directions, order \(V\) counterclockwise, and let \(P_j\) be the Fourier projection onto the half-plane \(\{\xi\in\R^2:\,\xi\cdot v_j>0\}\), so that \(H_{v_j}=-i\left(2P_j-\mathrm{Id}\right)\) and it suffices to bound \(\sup_j|P_jf|\). Rotation by \(\pi/2\) preserves the cyclic order of \(\mathbb S^1\), so the lines \(v_1^\perp,\ldots,v_N^\perp\) occur in the same order as the directions themselves, and consecutive half-planes differ by the two-sheeted sector \(\Gamma_j\) bounded by \(v_j^\perp\) and \(v_{j+1}^\perp\). These sectors are pairwise disjoint. The operator \(\Delta_j\coloneqq P_{j+1}-P_j\) has multiplier supported in \(\Gamma_j\), so that
\[
 P_jf=P_1f+\sum_{i<j}\Delta_if
\]
expresses \(P_jf\) as \(P_1f\) plus partial sums of mutually orthogonal increments. Estimating \(P_1f\) separately, the Rademacher--Menshov theorem bounds the maximal partial sum by \(\log N\) times \(\|f\|_2\), and the upper bound follows.

Two features deserve emphasis. First, the argument uses only the cardinality of \(V\) and the ordering of \(\mathbb S^1\). The same sectors \(\Gamma_j\) underlie the almost-orthogonality proof in \S\ref{sec:AO-plane}; the absence of this ordering makes both arguments harder in higher dimensions; see \S\ref{sec:higher-arbitrary} and \S\ref{sec:algebraic-AO}. The second is that the argument uses the piecewise constant symbol of the Hilbert transform, which makes consecutive operators differ by a projection onto a sector. For a general H\"ormander--Mihlin symbol \(m(\xi\cdot v)\) no telescoping of this kind is available, and the elementary argument collapses. That
\begin{equation}\label{eq:HM-logN}
 \left\|T_{m,V}\right\|_{L^2(\R^2)}\lesssim\log N
\end{equation}
nevertheless holds for every H\"ormander--Mihlin symbol \(m\), with constant depending only on a fixed H\"ormander--Mihlin norm, is a deep theorem of Demeter \cite{ciprianpams}. We return to \eqref{eq:HM-logN} in \S\ref{sec:vf}, since one of its two ingredients is a theorem about vector fields. Finally, the Rademacher--Menshov argument is confined to \(L^2\), and says nothing at the weak-type endpoint nor for \(p\neq2\). Both are harder, and both are places where the structure of \(V\) returns.

%%%%%%%%%%%%%%%%%%%%%%%%%%%%%% SUBSECTION SUBSECTION SUBSECTION
\subsection{Lacunary directions}\label{sec:lac} Lacunary sets have historically played an important role in harmonic analysis. For instance, they arise in the theory of lacunary Fourier series and Sidon sets, see for example, \cite{Graf}*{\S3.6}, as well as in Littlewood--Paley theory, where lacunary separation leads to almost-orthogonality phenomena. In directional problems, lacunary sets are important because they rule out the radial and Kakeya-type counterexamples described in Section~\ref{sec:cex}.  A model lacunary set of directions is
\[
V_{\mathrm{lac}}=\left\{{(\cos (2^{-k}\pi), \sin (2^{-k}\pi))}:\, k\in\N\right\}\subset\mathbb S^1.
\]
The notion of lacunarity that we shall use is the one introduced in \cite{PR}, which has the advantage of making sense in every dimension.

%%%%%%%%%%%%%%%%%%%%%%%%%%%%%% DEFINITION DEFINITION DEFINITION
\begin{definition}[Parcet--Rogers lacunarity] Let \(\Omega\subset\mathbb S^{n-1}\), let \(d\coloneqq \dim\operatorname{span}(\Omega)\), and fix an orthonormal basis \((e_1,\ldots,e_d)\) of \(\operatorname{span}(\Omega)\). For each pair \(\sigma=(j,k)\) with \(1\leq j<k\leq d\), choose \(0<\lambda_\sigma<1\) and a sequence \(\{\theta_{\sigma,i}\}_{i\in\Z}\) with \(0<\theta_{\sigma,i+1}\leq\lambda_\sigma\theta_{\sigma,i}\), and set
\begin{equation}\label{eq:lac-dissection}
 \Omega_{\sigma,i}\coloneqq \left\{\omega\in\Omega:\ \theta_{\sigma,i+1}<\left|\frac{\omega\cdot e_k}{\omega\cdot e_j}\right|\leq\theta_{\sigma,i}\right\},
 \qquad
 \Omega_{\sigma,\infty}\coloneqq \Omega\cap\left(e_j^\perp\cup e_k^\perp\right).
\end{equation}
The collection \(\{\Omega_{\sigma,i}:\ 1\leq j<k\leq d,\ i\in\Z\cup\{\infty\}\}\) is a \emph{lacunary dissection} of \(\Omega\). A singleton is \emph{lacunary of order \(0\)}; for \(L\geq1\), the set \(\Omega\) is \emph{lacunary of order \(L\)} if it admits a lacunary dissection all of whose pieces are lacunary of order at most \(L-1\), the constants \(\lambda_\sigma\) being fixed throughout the recursion. Finally, \(\Omega\) is \emph{lacunary} if it is a finite union of sets of finite order.
\end{definition}
%%%%%%%%%%%%%%%%%%%%%%%%%%%%%% DEFINITION DEFINITION DEFINITION

When \(d=2\) there is a single pair \(\sigma\), so that \eqref{eq:lac-dissection} partitions \(\Omega\) into arcs of lacunarily decreasing slope, and the definition coincides, up to finite unions, with the classical one of Sj\"ogren--Sj\"olin \cite{SjSj}. The model set \(V_{\mathrm{lac}}\) above is lacunary of order one.

The first results for lacunary directional operators date back to C\'ordoba and R. Fefferman \cite{CoFeflac}, who proved that $\M_{V_{\mathrm{lac}}}$ is of weak type $(2,2)$ in $\R^2$. Carbery \cite{Car}, simplifying their covering argument, upgraded this to strong type $(2,2)$, still by geometric means. Independently, and around the same time as \cite{CoFeflac}, a slightly weaker distributional estimate close to $L^2$ was proved by Str\"omberg in \cite{Stromblac}. Below $L^2$, no geometric proof is currently known. However, Nagel--Stein--Wainger proved in \cite{NSW} that $\M_{V_{\mathrm{lac}}}$ is of strong type $(p,p)$ for all $1<p\leq\infty$ by Fourier-analytic methods. In \cite{SjSj}, Sj\"ogren--Sj\"olin prove the corresponding bounds for all $V$ that are \emph{lacunary sets of finite order}, and also provide a general treatment of directional square functions associated to such sets of directions.

These results lead to a natural structural question: is lacunarity of $V$, in an appropriate sense, necessary for \(L^p\)-boundedness of \(\M_V\) in the plane? The answer is not known in full generality. Results of Katz \cite{Katzcantor}, Bateman--Katz \cite{BKcantor}, and more recent work of Hagelstein--Radillo-Murgu\'ia--Stokolos \cites{HRMSduke,HRMS} support the idea that boundedness of \(\M_V\) imposes lacunary structure on \(V\). While the result of Bateman \cite{Batduke} had been understood as settling this question, Hagelstein--Radillo-Murgu\'ia--Stokolos \cites{HRMSduke,HRMS} revealed a gap in that argument. Thus the characterization of sets of directions for which $\M_V$ is bounded on some $L^p(\R^2)$ for $p\in(1,\infty)$ remains open. As this survey was being completed, Kroc--Lee--Pramanik \cite{KLP} announced a resolution, in the form of a trichotomy supplying the Kakeya-type constructions that \cite{Batduke} was missing. That work appeared too late for us to take it into account, and we have left the discussion above as it stands.

For maximal directional singular integrals, although lacunary directions avoid Kakeya-type obstructions, the obstruction from \cite{Kara} shows that one cannot have uniform bounds over infinite sets of directions. However, for lacunary \(V\) as above, Demeter--Di Plinio proved in \cite{DDP} the sharp bound
\[
\sup_{\substack{V\subset V_{\mathrm{lac}} \\ \#V\le N}} \|T_{m,V}\|_{L^2(\R^2)}\lesssim \sqrt{\log N}.
\]
This result was extended to lacunary sets of finite order in \cites{ADPP,DPPIJM,DPPIMRN}. See also the discussion in Section~\ref{sec:higher_dimensions} for the higher dimensional case.

%%%%%%%%%%%%%%%%%%%%%%%%%%%%%% SECTION SECTION SECTION
\subsection{Equispaced sets}\label{sec:equi} Here \(V_\delta\subset \mathbb S^1\) denotes a maximal \(\delta\)-separated set of directions. We will refer to such sets as equispaced, or uniform, \(\delta\)-nets. Historically, directional operators associated with such nets were among the first to be extensively studied, because of their connection with the Nikodym and Kakeya maximal conjectures. C\'ordoba proved in \cite{CordAJM} the sharp single-scale bound
\[
\left\|\M_{V_\delta,1}\right\|_{L^2(\R^2)}\simeq \sqrt{\log\frac{1}{\delta}}, \qquad 0<\delta\leq 1.
\]
The corresponding weak-type \(L^2(\R^2)\to L^{2,\infty}(\R^2)\) bound is of the same order and was proved by Str\"omberg in \cite{Strequi}. For singular integrals, equispaced configurations enjoy the essentially optimal weak type estimate
\begin{equation}\label{eq:DDP-weak}
 \|H_{V_\delta}\|_{L^2(\R^2)\to L^{2,\infty}(\R^2)} \lesssim \sqrt{\log \frac{1}{\delta}} \left(\log\log \frac{1}{\delta}\right)^6, \qquad 0<\delta\leq 1,
\end{equation}
due to Demeter--Di Plinio  \cite{DDP}. Estimate \eqref{eq:DDP-weak} is obtained in \cite{DDP} for a wider class of direction sets, there called \emph{Vargas sets}, which contains the equispaced nets and also the finite truncations of the Cantor set. The proof essentially contains a version of the proof of Carleson's theorem on convergence of Fourier series, together with further time-frequency ingredients of  multiparameter flavor. We will return to this connection in \S\ref{sec:vf}.

%%%%%%%%%%%%%%%%%%%%%%%%%%%%%% SUBSECTION SUBSECTION SUBSECTION
\subsection{General sets of directions}\label{sec:general} For a finite set $V\subset\mathbb S^1$ of arbitrary structure, one can quantify the failure of uniform boundedness as $N=\#V\to\infty$. For maximal averages the absence of a preferred scale makes the questions substantially harder. The elementary counting and inductive arguments that rely on the uniform density of equispaced sets are no longer available, and new methods are needed. This difficulty is reflected in the chronology: more than two decades after C\'ordoba's uniform result, Katz proved in \cite{KatzDuke} the sharp cardinality bounds
\begin{equation} \label{e:secondbd}
 \|\M_V\|_{L^2(\R^2)\to L^2(\R^2)}\lesssim \log N, \qquad \|\M_V\|_{L^2(\R^2)\to L^{2,\infty}(\R^2)}\lesssim \sqrt{\log N}.
\end{equation}
Both bounds are best possible, and the first follows from the second via the interpolation trick of Str\"omberg \cite{Strequi}. Katz had previously proved in \cite{Katzsingle} the sharp $\sqrt{\log N}$ bound for the single-scale operator $\M_{V,1}$ on $L^2(\R^2)$. The sharp $L^2(\R^2)$ bound for $\M_V$ was reproved by Alfonseca, Soria and Vargas via an \emph{almost orthogonality principle} in \cite{ASVmounthol}. Furthermore, the authors in \cite{ASVmounthol} combine a weak-$L^2$ almost orthogonality principle with an argument of A. Carbery to give
\[
\sup_{\substack{V\subset \mathbb S^1\\\#V\leq N}}  \|\M_{V,V^\perp}\|_{L^2(\R^2)\to L^{2,\infty}(\R^2)}\lesssim \sqrt{\log N} (\log \log N),
\]
where $\M_{V,V^\perp}$ denotes the maximal operator with respect to rectangles in $\R^2$ with sides pointing along the sets $V,V^\perp$.  

For the maximal directional Hilbert transform the strong \(L^2\) bound is \eqref{eq:HV-RM}, and requires nothing. The situation for \(p\neq2\) is different. The logarithmic bound
\begin{equation}\label{eq:DDP-Lp}
 \sup_{\substack{V\subset\mathbb S^1\\ \#V\leq N}}\left\|T_{m,V}\right\|_{L^p(\R^2)}\lesssim_p\log N,\qquad 2\leq p<\infty,
\end{equation}
valid for an arbitrary Calder\'on--Zygmund convolution kernel on \(\R\) and for \(V\) of arbitrary structure, was obtained in \cite{DDP} by leveraging the previously known case $p=2$ \cite{ciprianpams}, together with \eqref{e:secondbd} and Hunt's exponential refinement of the Coifman-Fefferman good-$\lambda$ inequality between $T_v$ and $\mathrm{M}_v$. Estimate \eqref{eq:DDP-Lp} is known to be sharp for $p=2$ only. In the case of the Hilbert transform along equispaced set $V$ of $N$ directions, interpolation of \eqref{eq:DDP-Lp} with \eqref{eq:DDP-weak} with $\delta\sim N^{-1}$ gives
\begin{equation}\label{eq:DDP-Lp1}
\left\|H_{V}\right\|_{L^p(\R^2)}\lesssim_p(\log N)^{1-\frac1p},\qquad 2<p<\infty,
\end{equation}
up to iterated logarithmic terms. 

It is reasonable to conjecture that \eqref{eq:DDP-Lp1} holds for general multipliers and configuration sets. By interpolation with \eqref{eq:DDP-Lp}, this would follow from a uniform weak type bound of order $\sqrt{\log N}$ for $T_{m,V}$, which is open even for the Hilbert transform; see \ref{op:weak}.

%%%%%%%%%%%%%%%%%%%%%%%%%%%%%% SUBSECTION SUBSECTION SUBSECTION
\subsection{The almost-orthogonality principle in the plane} \label{sec:AO-plane} Theorem \ref{thm:AO-plane} below is due to Alfon\-seca--Soria--Vargas \cite{ASV}, whose proof is spatial and of $TT^*$ type. Its $L^p$ counterpart is in Alfonseca \cite{Alf_AO}. We sketch instead the Fourier-analytic proof of \cite{DPPsubspace}*{\S3}, which makes the constants transparent and is the exact planar shadow of the algebraic principle of \S\ref{sec:algebraic-AO}. Its two ingredients are the lemmas below, stated in \(\R^n\) for use in higher dimensions.

Throughout, an unsubscripted \(\M\) denotes the Hardy--Littlewood maximal operator on \(\R^n\), the subscripted \(\M_{v,r},\M_{V,r},\M_V\) retaining the meaning assigned above. For \(v\in\mathbb S^{n-1}\) and \(0<\delta\leq1\) we write
\begin{equation}\label{eq:plate}
\begin{split} 
 T_\delta(v)&\coloneqq \left\{y\in\R^n:\, |y\cdot v|<1,\ \left|y-(y\cdot v)v\right|<\delta\right\},
\\ 
 \Gamma_{v,\delta}&\coloneqq \left\{\xi\in\R^n\setminus\{0\}:\, |\xi\cdot v|<2^{-2}\delta|\xi|\right\},
 \end{split}
\end{equation}
so that \(rT_\delta(v)\) is a tube of length \(r\) along \(v\) and thickness \(r\delta\) transversally, while \(\Gamma_{v,\delta}\) is the dual frequency region: a two-sided conical neighborhood of the hyperplane \(v^\perp\). Abusing notation, \(\Gamma_{v,\delta}\) denotes also the rough Fourier restriction \(\widehat{\Gamma_{v,\delta}f}\coloneqq \ind_{\Gamma_{v,\delta}}\widehat f\). Fixing a radial \(\Phi\in\mathcal S(\R^n)\) with \(\ind_{B(0,1)}\leq\Phi\leq\ind_{B(0,2)}\), we split
\begin{equation}\label{eq:lowhigh}
A_{v,r}=A^{>\delta}_{v,r}+A^{<\delta}_{v,r}, \qquad \widehat{A^{>\delta}_{v,r}f}(\xi)\coloneqq \Phi\!\left(\tfrac{r\delta\xi}{4}\right)\widehat{A_{v,r}f}(\xi),\qquad \xi \in \R^n.
\end{equation}
The superscripts refer to spatial rather than frequency scales: \(A^{>\delta}_{v,r}\) averages at the scale of the plate \(rT_\delta(v)\), and therefore sees only frequencies \(|\xi|\lesssim(r\delta)^{-1}\). The two lemmas below quantify the resulting dichotomy: below that threshold the average along \(v\) and at scale $r$ does not distinguish \(v\) from a direction \(\delta\)-nearby, while above it the average is supported in \(\Gamma_{v,\delta}\).

Finally, it is convenient to allow an arbitrary set of scales \(S\subset(0,\infty)\), writing \(\M_{V,S}f\coloneqq\sup_{v\in V,\,r\in S}\M_{v,r}f\) and \(A_{V,S}f\coloneqq\sup_{v\in V,\,r\in S}|A_{v,r}f|\), so that \(S=\{r\}\) and \(S=(0,\infty)\) recover \(\M_{V,r}\) and \(\M_V\). By \eqref{eq:smooth-rough} and scaling,
\begin{equation}\label{eq:scales-equiv}
 \left\|\M_{V,S}\right\|_{L^2(\R^n)}\simeq \sup_{k\in\Z}\left\|A_{V,2^kS}\right\|_{L^2(\R^n)},
 \qquad 2^kS\coloneqq \left\{2^kr:\,r\in S\right\},
\end{equation}
uniformly in \(V\) and \(S\). The first lemma is the switch mechanism just described. At spatial scales coarser than the plate \(rT_\delta(u)\), the average along \(v\) may be replaced by the average along any \(u\) with \(|u-v|\leq\delta\), at the price of a factor \(|u-v|/\delta\).

%%%%%%%%%%%%%%%%%%%%%%%%%%%%%% LEMMA LEMMA LEMMA
\begin{lemma}[Switch lemma]\label{lem:switch} Let \(r>0\), let \(u\neq v\in\mathbb S^{n-1}\), and let \(0<\delta\leq1\) satisfy \(|u-v|\leq4\delta\). Then, for all \(x\in\R^n\),
\begin{equation}\label{eq:switch}
\begin{split}
 \left|A^{>\delta}_{v,r}f(x)\right|+\frac{\delta}{|u-v|}\left|A^{>\delta}_{v,r}f(x)-A^{>\delta}_{u,r}f(x)\right| &\lesssim \sum_{k\geq0}2^{-kn}\avgint_{x+2^krT_\delta(u)}|f|
 \\
 &\lesssim \sum_{k\geq0}2^{-kn}\M\left[\M_{u,2^kr}f\right](x).
\end{split}
\end{equation}
\end{lemma}
%%%%%%%%%%%%%%%%%%%%%%%%%%%%%% LEMMA LEMMA LEMMA

%%%%%%%%%%%%%%%%%%%%%%%%%%%%%% PROOF PROOF PROOF 
We indicate the mechanism and send to \cite{DPPsubspace}*{Lemma 2.4} for the details. The first inequality in \eqref{eq:switch} comes from a kernel bound for the multiplier \(\Phi(r\delta\xi/4)[\widehat\varphi(r\,\xi\cdot v)-\widehat\varphi(r\,\xi\cdot u)]\) of \(A^{>\delta}_{v,r}-A^{>\delta}_{u,r}\), and two features of it are responsible. The first is its support: \(\Phi(r\delta\,\cdot/4)\) forces \(|\xi|\leq 8(r\delta)^{-1}\) while \(\widehat\varphi(r\,\xi\cdot v)\neq0\) forces \(|\xi\cdot v|\leq (2r)^{-1}\), and as \(|u-v|\leq4\delta\) both constraints persist with \(u\) in place of \(v\), so that the multiplier lives on a box of dimensions \(\sim r^{-1}\) along \(u\) and \(\sim(r\delta)^{-1}\) in the orthogonal directions, which is the dual box of the plate \(rT_\delta(u)\) of \eqref{eq:plate}. The second is its size, which is \(\lesssim|u-v|/\delta\) there by the mean value theorem. Differentiating once in the direction \(u\) costs a factor \(r\) and once in an orthogonal direction a factor \(r\delta\), so repeated integration by parts turns the two features into a kernel dominated by \(|u-v|/\delta\) times an \(L^1\)-normalized bump adapted to \(rT_\delta(u)\); dominating it along the dilates \(2^krT_\delta(u)\) bounds the second term on the left, and the first term is the case \(u=v\) of the same computation, the mean value factor being replaced by \(\|\widehat\varphi\|_\infty\). The second inequality is the pointwise bound \(\avgint_{x+\rho T_\delta(u)}|f|\lesssim \M[\M_{u,\rho}f](x)\), which is simply the result of fattening the one-dimensional average \(\M_{u,\rho}f\) transversally to \(u\), at scale \(\rho\delta\).
%%%%%%%%%%%%%%%%%%%%%%%%%%%%%% PROOF PROOF PROOF

The second lemma is the complementary statement, and is immediate.
%%%%%%%%%%%%%%%%%%%%%%%%%%%%%% LEMMA LEMMA LEMMA
\begin{lemma}\label{lem:cone}
Let \(r>0\), \(v\in\mathbb S^{n-1}\) and \(0<\delta\leq1\). Then \(A^{<\delta}_{v,r}f=A^{<\delta}_{v,r}\Gamma_{v,\delta}f\).
\end{lemma}
%%%%%%%%%%%%%%%%%%%%%%%%%%%%%% LEMMA LEMMA LEMMA

%%%%%%%%%%%%%%%%%%%%%%%%%%%%%% PROOF PROOF PROOF
\begin{proof} Let \(\xi\) belong to the Fourier support of \(A^{<\delta}_{v,r}f\). Then \(\Phi(r\delta\xi/4)\neq1\) in \eqref{eq:lowhigh}, whence \(|\xi|\geq 4(r\delta)^{-1}\), and \(\widehat\varphi(r\,\xi\cdot v)\neq0\), whence \(|\xi\cdot v|\leq (2r)^{-1}\leq 2^{-3}\delta|\xi|\). Therefore \(\xi\in\Gamma_{v,\delta}\).
\end{proof}
%%%%%%%%%%%%%%%%%%%%%%%%%%%%%% PROOF PROOF PROOF

We come to the almost orthogonality principle itself.
%%%%%%%%%%%%%%%%%%%%%%%%%%%%%% THEOREM THEOREM THEOREM
\begin{theorem}[Planar almost orthogonality]\label{thm:AO-plane} There is an absolute constant \(C>0\) such that the following holds. Let \(S\subset(0,\infty)\), let \(\Omega=\{\omega_1,\ldots,\omega_{N+1}\}\subset\mathbb S^1\) be ordered counterclockwise, and for \(1\leq j\leq N\) let \(\Omega_j\) be the open arc with endpoints \(\omega_j,\omega_{j+1}\). If \(\,V_j\subset\Omega_j\) for each \(j\) and \(V\coloneqq\bigcup_{j}V_j\), then
\begin{equation}\label{eq:AO-plane}
 \left\|A_{V,S}\right\|_{L^2(\R^2)} \leq C\left\|A_{\Omega,S}\right\|_{L^2(\R^2)} + \max_{1\leq j\leq N}\left\|A_{V_j,S}\right\|_{L^2(\R^2)}.
\end{equation}
\end{theorem}
%%%%%%%%%%%%%%%%%%%%%%%%%%%%%% THEOREM THEOREM THEOREM

%%%%%%%%%%%%%%%%%%%%%%%%%%%%%% PROOF PROOF PROOF
\begin{proof} Throughout, \(C>0\) is an absolute constant which may vary at each occurrence. There is no loss in generality with assuming that each  \(\Omega_j\) is contained in the  same arc of length at most \(\frac\pi4\).  Let \(\Gamma_j\) be the two-sheeted cone bounded by \(\omega_j^\perp\) and \(\omega_{j+1}^\perp\), denoting by \(\Gamma_j\) also the rough Fourier restriction to it. These cones are pairwise disjoint, the one point at which the plane is special, and therefore
\begin{equation}\label{eq:AO-orth}
 \sum_{j=1}^N\left\|\Gamma_jf\right\|_2^2\leq \|f\|_2^2.
\end{equation}
Fix \(j\) and \(v\in V_j\). Let \(\omega(v)\in\{\omega_j,\omega_{j+1}\}\) be a closest endpoint of the arc, set \(\delta(v)\coloneqq|v-\omega(v)|\leq1\). Since \(v\) lies at distance at least \(\delta(v)\) from both endpoints of \(\Omega_j\), its perpendicular \(v^\perp\) lies at distance at least \(\delta(v)\) from both lines bounding \(\Gamma_j\), and therefore the whole cone \(\Gamma_{v,\delta(v)}\) fits inside \(\Gamma_j\), 
\begin{equation}\label{eq:AO-cone-incl}
 \Gamma_{v,\delta(v)}\subset\Gamma_j,\qquad v\in V_j.
\end{equation}
Combining Lemma~\ref{lem:cone} with \eqref{eq:AO-cone-incl} yields \(A^{<\delta(v)}_{v,r}f=A^{<\delta(v)}_{v,r}\Gamma_jf\), so that, splitting both \(f\) and \(\Gamma_jf\) according to \eqref{eq:lowhigh},
\[
\left|A_{v,r}f\right| \leq \left|A^{>\delta(v)}_{v,r}f\right|+\left|A^{<\delta(v)}_{v,r}\Gamma_jf\right| \leq \left|A_{v,r}\Gamma_jf\right|+\left|A^{>\delta(v)}_{v,r}f\right|+\left|A^{>\delta(v)}_{v,r}\Gamma_jf\right|.
\]
Applying \eqref{eq:switch} of Lemma~\ref{lem:switch} twice, once to \(f\) and once to \(\Gamma_jf\), with \(u=\omega(v)\) and \(\delta=\delta(v)\),  
\[
\left|A_{v,r}f\right| \leq \left|A_{v,r}\Gamma_jf\right|+ C\sum_{k\geq0}2^{-2k}\left(\M\left[\M_{\omega(v),2^kr}f\right]+\M\left[\M_{\omega(v),2^kr}\Gamma_jf\right]\right).
\]
Taking the supremum over \(v\in V_j\) and \(r\in S\), then over \(1\leq j\leq N\), and finally \(L^2(\R^2)\) norms,
\begin{equation}\label{eq:AO-three}
 \left\|A_{V,S}f\right\|_2 \leq \left\|\max_{j}\left|A_{V_j,S}\Gamma_jf\right|\right\|_2 +C\sup_{k\geq0}\left\|\M\circ\M_{\Omega,2^kS}f\right\|_2 +C\sup_{k\geq0}\left\|\max_{j}\M\circ\M_{\Omega,2^kS}\left[\Gamma_jf\right]\right\|_2 .
\end{equation}
For the second term, the \(L^2(\R^2)\)-boundedness of the Hardy--Littlewood maximal operator followed by \eqref{eq:scales-equiv} gives
\[
\sup_{k\geq0}\left\|\M\circ\M_{\Omega,2^kS}\right\|_{L^2(\R^2)} \leq C\sup_{k\geq0}\left\|\M_{\Omega,2^kS}\right\|_{L^2(\R^2)} \leq C\left\|A_{\Omega,S}\right\|_{L^2(\R^2)}.
\]
The first and the third are the same argument: for any family \(\{B_j\}\) of operators, bounding the maximum by the square function and invoking \eqref{eq:AO-orth},
\[
\left\|\max_{j}\left|B_j\Gamma_jf\right|\right\|_2 \leq \left(\sum_j\left\|B_j\Gamma_jf\right\|_2^2\right)^{\frac12} \leq \left(\max_j\left\|B_j\right\|_{L^2(\R^2)}\right)\|f\|_2 .
\]
Taking \(B_j=\M\circ\M_{\Omega,2^kS}\) bounds the third term by \(C\|A_{\Omega,S}\|_{L^2(\R^2)}\|f\|_2\), again by \eqref{eq:scales-equiv}, and taking \(B_j=A_{V_j,S}\) bounds the first by \(\max_j\|A_{V_j,S}\|_{L^2(\R^2)}\|f\|_2\). The latter is the source of the constant \(1\) in front of the second term of \eqref{eq:AO-plane}. Inserting these bounds into \eqref{eq:AO-three} completes the proof.
\end{proof}
%%%%%%%%%%%%%%%%%%%%%%%%%%%%%% PROOF PROOF PROOF

The recursion is now immediate, and recovers the sharp cardinality bound of Katz \cite{KatzDuke}.

%%%%%%%%%%%%%%%%%%%%%%%%%%%%%% COROLLARY COROLLARY COROLLARY
\begin{corollary}\label{cor:logN}
There holds $\displaystyle \sup_{\substack{V\subset\mathbb S^1\\ \#V\leq N}}\left\|\M_V\right\|_{L^2(\R^2)}\simeq \log N.$
\end{corollary}
%%%%%%%%%%%%%%%%%%%%%%%%%%%%%% COROLLARY COROLLARY COROLLARY

%%%%%%%%%%%%%%%%%%%%%%%%%%%%%% PROOF PROOF PROOF
\begin{proof} For \(N\geq1\) set
\[
 \mathsf A(N)\coloneqq \sup_{S\subset(0,\infty)}\ \sup_{\substack{V\subset\mathbb S^1\\ \#V\leq N}}\left\|A_{V,S}\right\|_{L^2(\R^2)} .
\]
Fix \(S\) and \(V\) with \(\#V\leq N\). Identifying antipodal directions we may assume \(V\) is contained in a closed half-circle, and we take \(\Omega=\{\omega_1,\omega_2,\omega_3\}\) with \(\omega_1,\omega_3\) the endpoints of that half-circle and \(\omega_2\) a median element of \(V\), so that the two arcs \(\Omega_1,\Omega_2\) determined by \(\Omega\) satisfy \(\#(V\cap\Omega_j)\leq\lceil N/2\rceil\). As \(\#\Omega=3\) and each single-direction maximal average is bounded on \(L^2(\R^2)\) by Fubini's theorem, \(\|A_{\Omega,S}\|_{L^2(\R^2)}\leq C\). Theorem~\ref{thm:AO-plane} therefore gives
\[
 \mathsf A(N)\leq C+\mathsf A\!\left(\left\lceil \tfrac N2\right\rceil\right),
\]
and iterating over the \(\lesssim\log N\) dyadic generations yields \(\mathsf A(N)\lesssim\log N\). The claimed upper bound follows by \eqref{eq:smooth-rough}. The matching lower bound is \eqref{eq:radial-multiscale-endpoint} with \(n=2\) and \(\delta=N^{-1}\). See also \cite{ciprianpams}*{Proposition 3.3}.
\end{proof}
%%%%%%%%%%%%%%%%%%%%%%%%%%%%%% PROOF PROOF PROOF

%%%%%%%%%%%%%%%%%%%%%%%%%%%%%% SUBSECTION SUBSECTION SUBSECTION
\subsection{Vector fields and the Zygmund and Stein conjectures}\label{sec:vf} Let \(v:\R^2\to \mathbb S^1\) be a measurable vector field. The corresponding truncated maximal average along directions given by $v(\cdot)$ is
\begin{equation}
 \M_{v(\cdot),s}f(x)\coloneqq \sup_{0<h<s}\frac1{2h}\int_{-h}^h |f(x+tv(x))|\,\d t,\qquad x\in \R^2.
\end{equation}
Zygmund and Stein proposed that sufficient smoothness of \(v\) should rule out Kakeya-type counterexamples. However, H\"older smoothness \(C^\alpha\) with \(\alpha<1\) is still not enough: one can construct \(C^\alpha\) vector fields for which Kakeya-type counterexamples persist; see \cites{BatRMI,DPGTZK}. The expectation in the  Lipschitz case \(\alpha=1\) is that this is no longer possible if the truncation is suitably chosen. More precisely, Zygmund's conjecture predicts that if \(v\) is Lipschitz and the scales are truncated at \(s\simeq \|v\|_{\operatorname{Lip}}^{-1}\), then \(\M_{v(\cdot),s}\) should be of weak type \((2,2)\) and thus also of strong type $(p,p)$ for all $2<p\leq \infty$. The analogous conjecture for the truncated directional Hilbert transform
\begin{equation}
 H_{v(\cdot),s}f(x)\coloneqq \frac{1}{\pi}\pv\int_{-s}^s f(x+tv(x))\frac{\d t}{t},\qquad x\in\R^2,
\end{equation}
is usually attributed to Stein. 

A decisive result of Lacey--Li \cite{LaceyLi} is the single-annulus estimate. Let \(P_0\) be a smooth frequency projection onto \(\{\xi\in\R^2:\,|\xi|\simeq1\}\). Then, for any measurable vector field $v:\R^2\to \mathbb S^1$,
\begin{equation}\label{eq:LaceyLi}
 \|H_{v(\cdot)}P_0 f\|_{L^{2,\infty}(\R^2)}\lesssim \|f\|_{L^2(\R^2)},\qquad  \|H_{v(\cdot)}P_0 f\|_{L^p(\R^2)} \lesssim \|f\|_{L^p(\R^2)},\quad 2<p<\infty.
\end{equation}
After localization to one annulus and a sufficiently thin angular sector, the half-plane projection \(\{\xi:\xi\cdot v(x)>0\}\) takes a one-dimensional Carleson-type form, with modulation parameter determined by the vector field. The phase-plane analysis in \cite{LaceyLi} is correspondingly a vector-field analogue of that behind Carleson's theorem.
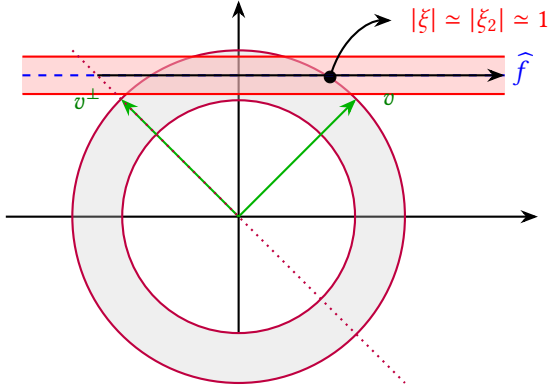
\begin{figure}[ht]
\centering
\begin{minipage}[c]{0.50\linewidth}
\centering
\begin{tikzpicture}[scale=2.2, >=Stealth]
% Axes
\draw[->, thick] (-1.4,0) -- (1.8,0);
\draw[->, thick] (0,-0.7) -- (0,1.3);

% Annulus
\path[fill=gray!40, opacity=0.3, even odd rule]
  (0,0) circle (1)
  (0,0) circle (0.7);
\draw[thick, purple] (0,0) circle (0.7);
\draw[thick, purple] (0,0) circle (1);

% Band bounds
\def\ymin{0.7375}
\def\ymax{0.9625}

% Horizontal band
\fill[red!30, opacity=0.5] (-1.3,\ymin) rectangle (1.6,\ymax);
\draw[red, thick] (-1.3,\ymax) -- (1.6,\ymax);
\draw[red, thick] (-1.3,\ymin) -- (1.6,\ymin);

% Band/frequency label
\node[red] at (1.45,1.18) {\footnotesize $|\xi|\simeq|\xi_2|\simeq 1$};

% Curved pointer from band \cap annulus to the label
\filldraw[black] (0.55,0.84) circle (1pt);
\draw[->, thick] (0.55,0.84) .. controls (0.62,1.08) and (0.78,1.18) .. (0.92,1.18);

% Support line
\draw[blue, thick, dashed] (-1.3,0.85) -- (1.6,0.85);
\node[blue] at (1.7,0.89) {\footnotesize $\widehat f$};

% Direction vectors
\draw[->, green!70!black, thick] (0,0) -- ({1/sqrt(2)},{1/sqrt(2)});
\draw[->, green!70!black, thick] (0,0) -- ({-1/sqrt(2)},{1/sqrt(2)});
\node[green!50!black] at ({1/sqrt(2)+0.2},{1/sqrt(2)}) {\footnotesize $v$};
\node[green!50!black] at ({-1/sqrt(2)-0.2},{1/sqrt(2)}) {\footnotesize $v^\perp$};

% Guides
\draw[purple, dotted, thick] (-1,1) -- (1,-1);
\draw[->, thick] (-0.85, 0.85) -- (1.6, 0.85);

\end{tikzpicture}
\end{minipage}\hfill
\begin{minipage}[c]{0.50\linewidth}
\caption{Single-annulus localization and the Carleson operator:
the intersection of \(v^\perp\) with the frequency support
\(|\xi|\simeq|\xi_2|\simeq1\) determines the variable modulation parameter.}
\label{fig:lacey-li-carleson}
\end{minipage}
\end{figure}

The same mechanism gives the sharp $L^2$ estimate \eqref{eq:HM-logN} for general H\"ormander--Mihlin directional multipliers and an arbitrary finite set of directions. In Demeter's proof \cite{ciprianpams}, the Chang--Wilson--Wolff inequality \cite{CWW}, already used for maximal Fourier multipliers in \cite{GHS}, replaces the supremum over the \(N\) directions by a square function over annuli at a cost of \(\sqrt{\log N}\). On each annulus the Lacey--Li theorem gives a weak $L^2$ bound, which the interpolation trick of Str\"omberg \cite{Strequi} upgrades to strong type at a further cost of \(\sqrt{\log N}\); summing over the orthogonal annuli, the two losses combine to \(\log N\).

Passing from a single annulus to \(H_{v(\cdot)}\) itself requires summing the frequency scales. By the Littlewood--Paley diagonalization theorem of \cite{DPGTZK} for bi-Lipschitz perturbations of the identity on the line, when \(v\) is Lipschitz a Littlewood--Paley decomposition may be inserted into \(H_{v(\cdot)}\) and removed again at no cost. Stein's conjecture for Lipschitz fields thus reduces, without loss, to the vector-valued single-annulus bound of \ref{op:vv} for such fields. The difficulty is that the time-frequency analysis behind \cite{LaceyLi} works one annulus at a time, and its exceptional sets at different frequency scales are not compatible under \(\ell^2\) summation. 

The Littlewood--Paley summation can be handled in the one-variable setting, with no regularity assumption on the field. Bateman \cite{BatRMI} proved single-annulus \(L^p\) estimates for all \(p\in(1,\infty)\), and Bateman--Thiele \cite{BaTh} proved the unrestricted estimate for \(p\in(3/2,\infty)\). The diagonalization theorem of \cite{DPGTZK} recovers the latter from the former. Guo \cite{GuoAPDE} extended the \(L^2\) theory to fields constant on a family of Lipschitz curves.

Both conjectures have been settled for real analytic vector fields, by Bourgain \cite{Bourgain89} for the maximal operator and by Stein--Street \cite{SteinStreet} for the Hilbert transform. Bourgain's argument separates two ingredients. The first is a regularity assumption, roughly $v\in C^{2}$. The second is a geometric nondegeneracy condition, excluding straight integral curves of the vector field. For smooth fields under hypotheses of this kind see also \cite{CNSW}, and \cite{GuoJGA} for a geometric proof of Bourgain's $L^2$ estimate.

The memoir of Lacey--Li \cite{LaceyLiMem} relates the conjecture to a geometric maximal operator. Call the \emph{interval of uncertainty} of a rectangle $R$ the subarc $\mathrm{EX}(R)\subset \mathbb S^1$ centred at the long direction of $R$, of length equal to the eccentricity of $R$. For a Lipschitz field $v$, restrict attention to rectangles of length at most $(100\|v\|_{\mathrm{Lip}})^{-1}$, to prevent the feasibility of Kakeya-type constructions,  and write \[\mathrm V(R)\coloneqq R\cap v^{-1}(\mathrm{EX}(R))\] for the portion of $R$ on which the field points into $R$'s own interval of uncertainty. The \emph{Lipschitz Kakeya maximal function} is
\begin{equation}\label{eq:lipschitz-kakeya}
 \mathrm M_{v,\delta}f(x)\coloneqq \sup_{|\mathrm V(R)|\geq\delta|R|}\frac{\ind_R(x)}{|R|}\int_R|f| ,
 \qquad 0<\delta<1 .
\end{equation}
In \eqref{eq:lipschitz-kakeya} the eccentricity of $R$ is left free, and what is fixed instead is the proportion $\delta$ of $R$ on which $v$ is aligned with $R$. Where the classical Kakeya maximal function of \S\ref{sec:cex} must fix the eccentricity and pays a logarithm in it, here one has
\begin{equation}\label{eq:LK-weak2}
 \left\|\mathrm M_{v,\delta}\right\|_{L^2(\R^2)\to L^{2,\infty}(\R^2)}\lesssim\delta^{-\frac12} ,
\end{equation}
with implied constant independent of the Lipschitz field. Estimate \eqref{eq:LK-weak2} is one of the unconditional main results of \cite{LaceyLiMem}.

The theorem on Stein's conjecture is then conditional, and states exactly what is missing. Suppose $\mathrm M_{v,\delta}$ maps $L^p$ to $L^{p,\infty}$ with norm $\lesssim\delta^{-N}$ for some $p<2$ and some finite $N$. Then the single-annulus operator of \eqref{eq:LaceyLi} is bounded on $L^2(\R^2)$ uniformly in the annulus, and if $v$ has $1+\eta$ derivatives the untruncated operator obeys $\|H_{v(\cdot)}\|_{L^2(\R^2)}\lesssim(1+\log\|v\|_{C^{1+\eta}})^2$. Whether the hypothesis holds for every Lipschitz field is the principal conjecture of \cite{LaceyLiMem}; see \ref{op:lipkakeya}. Bourgain's geometric non-flatness condition does imply a suitable Lipschitz Kakeya estimate, so that this framework recovers the results on analytic vector fields.

%%%%%%%%%%%%%%%%%%%%%%%%%%%%%% SECTION SECTION SECTION
\section{Higher-dimensional directional operators}\label{sec:higher_dimensions} The Kakeya \emph{set} conjecture in \(\R^3\) has been, remarkably and rather recently, solved positively by Wang--Zahl \cite{WangZahl}. \(\M_{V,1}\) entail the maximal Nikodym conjecture. A bound of the form
\begin{equation}\label{eq:nikodym-max}
 \left\|\M_{V,1}\right\|_{L^n(\R^n)}\lesssim_\eps(\#V)^\eps\qquad\forall\eps>0,\qquad V\subset\mathbb S^{n-1},
\end{equation}
would imply the maximal Nikodym conjecture and, by the equivalence \cite{TaoBRR}, the maximal Kakeya conjecture as well.

The Kakeya \emph{set} conjecture in \(\R^3\) has been, remarkably and rather recently, solved in the positive  by Wang--Zahl \cite{WangZahl}. It is now known that every Kakeya set in $\R^3$ has full Hausdorff dimension. However, the maximal conjecture remains open in every dimension \(n\geq3\). What is available in higher dimensions is, on the one hand, sharp information when the directions are constrained to an algebraic set, and on the other hand sharp information for maximal averages along codimension one subspaces. We survey both, together with the higher-dimensional forms of lacunarity.

%%%%%%%%%%%%%%%%%%%%%%%%%%%%%% SUBSECTION SUBSECTION SUBSECTION
\subsection{Lacunarity in higher dimensions}\label{sec:higher-lac} In the plane, lacunarity is expected to characterize the direction sets along which differentiation holds; see \S\ref{sec:lac} for the status of this characterization. The correct higher-dimensional notion is due to Parcet--Rogers \cites{PRproc,PR}, who defined finite-order lacunary subsets of \(\mathbb S^{n-1}\) by an iterated cone construction and proved that \(\M_V\) is bounded on \(L^p(\R^n)\) for every \(p\in(1,\infty]\) when \(V\) is of this type, together with a converse under a mild separation hypothesis. Antecedents in special configurations are due to Nagel--Stein--Wainger \cite{NSW} and Carbery \cite{Car}.

At the level of singular integrals the sharp cardinality dependence is known in every dimension, for the maximal directional multiplier \(T_{m,V}\) of \eqref{eq:max-dir-mult}.
%%%%%%%%%%%%%%%%%%%%%%%%%%%%%% THEOREM THEOREM THEOREM
\begin{theorem}[\cite{ADPP}]\label{thm:lac-SI} Let \(V_{\mathrm{lac}}\subset\mathbb S^{n-1}\) be an infinite set of finite lacunary order, and let \(1<p<\infty\). Then
\begin{equation}
\sup_{\|m\|_{\mathrm{HM}(\R)}\leq1} \sup_{\substack{V\subset V_{\mathrm{lac}}\\ \#V\leq N}} \left\|T_{m,V}\right\|_{L^p(\R^n)} \simeq_{n,p,V_{\mathrm{lac}}}\sqrt{\log N}.
\end{equation}
\end{theorem}
%%%%%%%%%%%%%%%%%%%%%%%%%%%%%% THEOREM THEOREM THEOREM

The \(\sqrt{\log N}\) loss is of Chang--Wilson--Wolff type, as in the plane, and is sharp even for lacunary sets of order one. The proofs in dimensions two and three \cites{DPPIJM,DPPIMRN} use products of the two-dimensional inner and outer wedge multipliers of Parcet--Rogers \cite{PR} and control their maximal truncations. In higher dimensions, this approach requires a Littlewood--Paley square function in \(\binom n2\) parameters. Combining planar counterexamples in coordinate pairs forces a loss of \((\log N)^{\lceil n/2\rceil}\), preventing a sharp bound.

Instead, \cite{ADPP} covers the singular hyperplane \(\xi\cdot v=0\) by the exterior of a full-dimensional Nagel--Stein--Wainger cone \cite{NSW}. For lacunary directions, covering these exteriors by Parcet--Rogers wedges gives a bounded square function. The remaining maximal conical multiplier involves only one parameter, so a single application of Chang--Wilson--Wolff \cite{CWW} gives the desired loss. Higher lacunary orders are handled recursively through weighted
vector-valued estimates and extrapolation.
 
%%%%%%%%%%%%%%%%%%%%%%%%%%%%%% SUBSECTION SUBSECTION SUBSECTION
\subsection{Arbitrary finite sets of directions}\label{sec:higher-arbitrary} An arbitrary \(V\) has no distinguished scale for induction. The algebraic principle of \S\ref{sec:algebraic-AO} supplies the localization \eqref{eq:AO}, first for directions constrained to a variety.

%%%%%%%%%%%%%%%%%%%%%%%%%%%%%% THEOREM THEOREM THEOREM
\begin{theorem}[\cite{DPPAJM}]\label{thm:varieties} Let \(Z\subset\R^n\) be a real algebraic variety of dimension \(m<n\). Then, for all \(\eps>0\), there is $K \geq 1$ such that
\begin{equation}\label{eq:varieties}
\sup_{\substack{V\subset Z\cap\mathbb S^{n-1}\\ \#V\leq N^m}} \left\|\M_{V,1}\right\|_{L^2(\R^n)}\leq K N^{\frac{m-1}{2}+\eps}.
\end{equation}
The exponent is best possible up to $\eps$-loss. The constant $K$ depends only  on $\eps$ and on explicit algebraic parameters of $Z$ such as dimension $m$, degree and count.
\end{theorem}
%%%%%%%%%%%%%%%%%%%%%%%%%%%%%% THEOREM THEOREM THEOREM

The exponent depends on the dimension \(m\) of the carrying variety, not on \(n\). Taking \(Z=\mathbb S^{n-1}\) gives \(\frac{n-2}{2}\) for arbitrary directions; in the plane, the sharp bound is \(\sqrt{\log N}\), as recalled after \eqref{e:secondbd}. Antecedents are due to C\'ordoba \cite{CorGFA} for \(D\)-crossing smooth curves, with a logarithmic loss, and, in the case \(n=3\), \(m=2\), to Demeter \cite{cipriangreedy}.

In three dimensions, applying algebraic almost orthogonality to a polynomial partition improves the \(N^\eps\) loss in \eqref{eq:varieties} to an iterated logarithm.
%%%%%%%%%%%%%%%%%%%%%%%%%%%%%% THEOREM THEOREM THEOREM
\begin{theorem}[\cite{DPPsubspace}*{Theorem F}]\label{thm:R3}
For every positive integer \(k\) there holds
\begin{equation}
 \sup_{\substack{V\subset\mathbb S^2\\ \#V\leq N^2}}\left\|\M_{V,1}\right\|_{L^2(\R^3)}\lesssim_k N^{\frac12}\log^{[k]}N,
\end{equation}
where \(\log^{[1]}N\coloneqq\log(2+N)\) and \(\log^{[k]}N\coloneqq\log\left(2+\log^{[k-1]}N\right)\). The bound is sharp in \(N\) up to the iterated logarithm.
\end{theorem}
%%%%%%%%%%%%%%%%%%%%%%%%%%%%%% THEOREM THEOREM THEOREM

Kim--Pramanik \cite{KimPramanik} proved the corresponding bound for \(H_V\): for a variety \(Z\) of dimension \(m\) and degree \(D\), 
\begin{equation}\label{eq:KP}
 \sup_{\substack{V\subset Z\cap\mathbb S^{n-1}\\ \#V\leq N^m}}\left\|H_V\right\|_{L^2(\R^n)}\lesssim_{\eps,m,D}N^{\frac{m-1}{2}+\eps},
\end{equation}
sharp in \(N\) for every \(m\) up to the \(\eps\)-loss by an example of J.~Kim \cite{JKim}, and admitting the same iterated logarithmic refinement in \(\R^3\). The proof combines polynomial partitioning with almost orthogonality for \(H_V\), exploiting its piecewise constant symbol. This prevents a direct extension of \eqref{eq:KP} to arbitrary H\"ormander--Mihlin multipliers. 

For directions on a curve of degree \(D\) the sharp bound \(\lesssim D\log N\) goes back to Christ--Duoandikoetxea--Rubio de Francia \cite{CDRdF}, whose Theorem 2 contains the planar estimate \eqref{eq:HV-RM} as well.  

%%%%%%%%%%%%%%%%%%%%%%%%%%%%%% SUBSECTION SUBSECTION SUBSECTION
\subsection{Subspace averages and the Nikodym problem}\label{sec:subspace} Averages along lines are the case \(d=1\) of averages along \(d\)-dimensional subspaces, and the whole range \(1\leq d<n\) is natural. For \(\sigma\in\Gr(d,n)\), \(s>0\) and finite \(\Sigma\subset\Gr(d,n)\) put
\begin{equation}
 f_{s,\sigma}(x)\coloneqq\int_{sB_n\cap\sigma}f(x-y)\,\frac{\d y}{s^d}, \qquad
 \M_{\Sigma,S}f\coloneqq\sup_{s\in S}\sup_{\sigma\in\Sigma}|f|_{s,\sigma},
\end{equation}
with \(\d y\) the Lebesgue measure on \(\sigma\). These averages are thin, the plates \(sB_n\cap\sigma\) carrying no thickness, and no pointwise domination by the Hardy--Littlewood maximal operator is available. The sharp \(L^2\) growth in the cardinality is nevertheless known in every dimension and codimension.
%%%%%%%%%%%%%%%%%%%%%%%%%%%%%% THEOREM THEOREM THEOREM
\begin{theorem}[\cite{DPPsubspace}*{Theorem A}]\label{thm:subspace} Let \(1\leq d<n\). There holds
\begin{equation}\label{eq:subspace-single}
 \sup_{s>0}\ \sup_{\substack{\Sigma\subset\Gr(d,n)\\ \#\Sigma\leq N}}\left\|\M_{\Sigma,\{s\}}\right\|_{L^2(\R^n)} \lesssim N^{\frac{n-d-1}{2(n-d)}}\left(\log N\right)^{\frac{n-d}2},
\end{equation}
sharp in \(N\) when \(n=d+1\), and sharp up to the logarithmic factor in general.
\end{theorem}
%%%%%%%%%%%%%%%%%%%%%%%%%%%%%% THEOREM THEOREM THEOREM
Codimension one is the case in which \(L^2(\R^n)\) is critical rather than subcritical, so that the logarithm in \eqref{eq:subspace-single} is not an artifact. There the multiscale operator is controlled as well, and sharply: for finite \(\Sigma\subset\Gr(d,d+1)\),
\begin{equation}\label{eq:subspace-codim1}
 \sup_{\substack{\Sigma\subset\Gr(d,d+1)\\ \#\Sigma\leq N}} \left\|\M_{\Sigma,(0,\infty)}\right\|_{L^2(\R^{d+1})} \simeq\log N,
\end{equation}
with weak type norm \(\left(\log N\right)^{1/2}\) and \(L^p\) norms \(\left(\log N \right)^{1/p}\) for \(p>2\), all best possible; see \cite{DPPsubspace}*{Theorem B}. For \(d=1\) this returns Corollary \ref{cor:logN}, and the mechanism behind it, a directional Carleson embedding for \((n-1)\)-dimensional plates, is described in \S\ref{sec:square}.

Running alongside the subspace averages is the \((d,n)\)-Nikodym maximal function \(\mathrm N_\delta\), the supremum of the averages of \(|f|\) over the \(\delta\)-neighborhoods of the affine \(d\)-plates through a given point. As for \(d=1\), bounds for \(\mathrm N_\delta\) with \(\eps\)-losses in \(\delta\) yield lower bounds for the Hausdorff dimension of \((d,n)\)-Nikodym sets, and elementary counterexamples single out the critical index \(p_\circ=n-d+1\).
%%%%%%%%%%%%%%%%%%%%%%%%%%%%%% CONJECTURE CONJECTURE CONJECTURE
\begin{conjecture}[\cite{DPPsubspace}*{Conjecture G}]\label{conj:nikodym}
Let \(1\leq d<n\). There holds
\[
\left\|\mathrm N_\delta\right\|_{L^p(\R^n)}\lesssim 
\begin{cases}
\delta^{-\frac{n-d+1-p}{p}}, & 1<p<n-d+1,
\\[2pt]
\left(\log\delta^{-1}\right)^{\alpha(n,d,p)}, & p\geq n-d+1,
\end{cases}
\]
for some exponent \(\alpha(n,d,p)>0\).
\end{conjecture}
%%%%%%%%%%%%%%%%%%%%%%%%%%%%%% CONJECTURE CONJECTURE CONJECTURE

The conjecture is verified in codimension one, where \(p_\circ=2\); for every other pair \((d,n)\) the sharp \(L^2\) bound, subcritical there, is \cite{DPPsubspace}*{Theorem C}. The passage from \(d=1\) to \(d>1\) is not a formality. The equivalence of the maximal Kakeya and maximal Nikodym conjectures recorded above fails for \(d>1\), and the corresponding critical exponents differ. This is already visible in the geometry: it is not known whether \((d,n)\)-Besicovitch sets exist for any \(d>1\), while it is known that they do not exist when \(d>n/2\) \cite{Falconer80} or when \(2^{d-1}+d\geq n\) \cite{Bourgain91}, the latter improved to \(\tfrac73 2^{d-2}+d\geq n\) in \cite{Oberlin}. On the other hand, \((d,n)\)-Nikodym sets do exist for every \(1\leq d<n\); see \cite{Falconer86}. Accordingly, \(\delta\)-free bounds are not excluded for the \((d,n)\)-Kakeya operator, while they are excluded for \(\mathrm N_\delta\) and for \(\M_{\Sigma,\{s\}}\). See \cite{Mattila} for these facts.

%%%%%%%%%%%%%%%%%%%%%%%%%%%%%% SUBSECTION SUBSECTION SUBSECTION
\subsection{Directional square functions and the Carleson embedding}\label{sec:square} The disc multiplier is the rough frequency projection on a disc; it is unbounded on \(L^p(\R^2)\) for every \(p\neq2\). This is Fefferman's theorem \cite{FeffBall}, whose proof runs through a Besicovitch set, and it is the initial reason of interest in the Kakeya problem. A polygon with \(N\) sides is a finitary model of the disc, so its multiplier norms must grow with \(N\), and quantifying that growth means proving quantitative \(\ell^2\) estimates for families of directional or conical multipliers with essentially disjoint supports. The model case is that of frequency projections onto finitely overlapping cones, which we state directly in \(\R^n\). Let \(B\subset\R^n\) be a ball centered on \(\mathbb S^{n-1}\), let \(\varphi_B\) be a smooth cutoff with \(\ind_B\leq\varphi_B\leq\ind_{3B}\), and define the smooth restriction to the frequency cone subtended by \(B\),
\begin{equation}\label{eq:cone-proj}
 \widehat{C^\circ_Bf}(\xi)\coloneqq \varphi_B\!\left(\frac{\xi}{|\xi|}\right)\widehat f(\xi),
 \qquad \xi\in\R^n\setminus\{0\} .
\end{equation}
%%%%%%%%%%%%%%%%%%%%%%%%%%%%%% THEOREM THEOREM THEOREM
\begin{theorem}[\cite{ADPHPR}*{Theorem A}, \cite{DPPsubspace}*{Theorem J}]\label{thm:cones}
Let \(\mathcal B\) be a collection of \(N\) balls centered on \(\mathbb S^{n-1}\) with overlap \(\sum_{B\in\mathcal B}\ind_{3B}\leq K\). Then
\begin{equation}\label{eq:cones-strong}
 \left\|\left\{C^\circ_Bf\right\}\right\|_{L^p(\R^n;\,\ell^2_{\mathcal B})}\lesssim_{K,p,n}\left(\log N\right)^{\frac12-\frac1p}\|f\|_{L^p(\R^n)},\qquad 2\leq p<4,
\end{equation}
together with the restricted type analogue, valid for every measurable \(E\subset\R^n\),
\begin{equation}\label{eq:cones-restricted}
 \left\|\left\{C^\circ_B\left(f\ind_E\right)\right\}\right\|_{L^4(\R^n;\,\ell^2_{\mathcal B})}\lesssim\left(\log N\right)^{\frac14}|E|^{\frac14}\|f\|_{L^\infty(\R^n)} .
\end{equation}
The dependence on \(N\) is best possible in both estimates, in every dimension.
\end{theorem}
%%%%%%%%%%%%%%%%%%%%%%%%%%%%%% THEOREM THEOREM THEOREM
Outside the \(2\leq p\leq 4\) range, only trivial polynomial estimates are available, and the matching lower bounds come from the Besicovitch construction. The strong type endpoint \(p=4\) is open in every dimension, including the plane. For \(\{\Delta_j\}\) a collection of disjoint \(\delta\)-balls on \(\mathbb S^{n-1}\), Gan--Wu \cite{GanWu} prove the strong \(L^4(\R^n)\) estimate with \(\delta^{-\eps}\) in place of a power of the logarithm, and show this to be sharp in that formulation. Square functions of this type go back to Nagel--Stein--Wainger \cite{NSW}, where they are the engine behind the lacunary maximal theorem. For \(N\) uniformly spaced cones C\'ordoba \cite{CorGFA} obtained a logarithmic estimate with unspecified dependence by weighted methods. The corresponding statement for rough conical projections, obtained by bootstrapping Theorem \ref{thm:cones}, is \cite{ADPHPR}*{Theorem J}, at the cost of a doubly logarithmic correction.

As mentioned above, the application that motivates the seeking of the sharp exponent is the polygon multiplier. Writing \(T_{\mathcal P_N}\) for the Fourier restriction to a regular \(N\)-gon, a combination of Theorem \ref{thm:cones} with vector-valued directional estimates and with a sharp form of Meyer's lemma yields
\begin{equation}\label{eq:polygon}
 \left\|T_{\mathcal P_N}\right\|_{L^p(\R^2)}\lesssim\left(\log N\right)^{4\left|\frac12-\frac1p\right|},\qquad \tfrac43<p<4;
\end{equation}
see \cite{ADPHPR}*{Theorem B}. This is the currently best known bound, and it improves on C\'ordoba \cite{CordPoly}, whose implicit dependence necessarily exceeds \(\left(\log N\right)^{\frac54}\) near the endpoints of the range. It does not match the best known lower bound, which is the one for the constant in Meyer's lemma; see \ref{op:polygon}. Estimate \eqref{eq:polygon} concerns a single dilate. The supremum over dilates is a problem of Carleson type, discussed in \S\ref{sec:vf}, and goes through a finitary version of the spherical summation problem. For a convex polygon \(\mathcal P\subset\R^2\) containing the origin in its interior and \(\lambda>0\), let
\begin{equation}\label{eq:polygon-dilate}
 T_{\mathcal P,\lambda}f(x)\coloneqq \int_{\R^2}\widehat f(\xi)\,\ind_{\mathcal P}\left(\frac\xi\lambda\right)\mathrm e^{2\pi ix\cdot\xi}\,\d\xi,\qquad x\in\R^2,
\end{equation}
be the Fourier restriction to the \(\lambda\)-dilate of \(\mathcal P\). Bounds for \(\sup_{\lambda>0}|T_{\mathcal P,\lambda}f|\) on \(L^p(\R^2)\) entail almost everywhere convergence of the polygonal Fourier cutoffs, \(T_{\mathcal P,\lambda}f\to f\) as \(\lambda\to\infty\), for every \(f\in L^p(\R^2)\). This is the polygonal form of the Fourier summation problem which, for the disc, Carleson's theorem answers in dimension one. No logarithmic bound is known about it.  The first genuinely new case not covered by Carleson plus Littlewood-Paley is a sequence accumulating at a single dilate, such as \(\lambda_n=1-2^{-n}\), which unlike the full supremum does not obviously contain a proof of Carleson's theorem.
The Kakeya obstruction is absent altogether for the polygons whose slopes form a lacunary set. The transference principle of C\'ordoba--Fefferman \cite{CoFemult}, converts the uniform $L^{(p/2)'}$ bounds on  \(\M_V\)  into a $L^p$ bound for the single multiplier \(T_{\mathcal P,\lambda}\)  uniform in the number of sides. Whether the same holds for the supremum over dilates is not known. Both maximal variants are \ref{op:maxpolygon}.

What makes all of these estimates available at once is a single embedding theorem. Linearizing and passing to the adjoint reduces each of them to a balayage
\begin{equation}\label{eq:balayage}
 \mathrm T_{\mathcal Q}(a)\coloneqq \sum_{Q\in\mathcal Q}a_Q\frac{\ind_Q}{|Q|}
\end{equation}
of a sequence $\{a_Q\}_Q$ indexed by plates of varying orientation. The sequence may be assumed to satisfy a Carleson packing condition with constant $1$ only relative to subfamilies of plates sharing a common orientation. That directional Carleson condition already yields
\begin{equation}\label{eq:carleson-embedding}
 \left\|\mathrm T_{\mathcal Q}(a)\right\|_{L^2}\lesssim\left(\log N\right)^{\frac12}\left(\log\log N\right)^{\frac\gamma2}\Bigl(\sum_{Q}a_Q\Bigr)^{\frac12},
\end{equation}
with \(N\) the number of orientations and \(\gamma=0\) or \(1\) according as the plates form a one- or a two-parameter family in each fixed orientation; see \cite{ADPHPR}*{Theorem C} in the plane and \cite{DPPsubspace}*{Theorem I} for \((n-1)\)-dimensional plates, both descending from the technique of Katz \cite{KatzDuke} and Bateman \cite{Batduke}. Applied to the adjoint of the linearized maximal operator, \eqref{eq:carleson-embedding} returns \eqref{eq:subspace-codim1}. Coupled instead with time-frequency analysis, so that \(a_Q\) records the energy of a wave packet adapted to \(Q\), it returns \eqref{eq:cones-strong}. 

%%%%%%%%%%%%%%%%%%%%%%%%%%%%%% SUBSECTION SUBSECTION SUBSECTION
\subsection{Codimension-one directional multipliers}\label{sec:codim1} The two-dimensional multiplier \(m(\xi\cdot v)\) is a codimension one object: its symbol is constant along \(v^\perp\). In higher dimensions the natural codimension one analogue is therefore indexed not by directions but by hyperplanes. Let \(d=n-1\), let \(\sigma\in\Gr(d,n)\) and let \(v^\sigma_1,\ldots,v^\sigma_d\) be an orthonormal basis of \(\sigma\); given a family \(\{m_\sigma\}\) of H\"ormander--Mihlin multipliers on \(\R^d\), set
\begin{equation}
 T_{m,\sigma}f(x)\coloneqq\int_{\R^n}m_\sigma\left(\xi\cdot v^\sigma_1,\ldots,\xi\cdot v^\sigma_d\right)\widehat f(\xi)e^{2\pi ix\cdot\xi}\,\d\xi, \qquad T_{m,\Sigma}f\coloneqq\sup_{\sigma\in\Sigma}\left|T_{m,\sigma}f\right| .
\end{equation}
The model example is the vector of \(d\)-dimensional Riesz transforms, \(m_\sigma(\eta)=\eta/|\eta|\). The central result is a single-annulus theorem, which is the codimension one counterpart of the planar band-limited estimates of \S\ref{sec:vf}.

%%%%%%%%%%%%%%%%%%%%%%%%%%%%%% THEOREM THEOREM THEOREM
\begin{theorem}[\cite{BDPPR}]\label{thm:codim1} Let \(P_k\) be a smooth Littlewood--Paley projection onto \(|\xi|\simeq2^k\) and let \(\sigma(\cdot):\R^n\to\Gr(n-1,n)\) be measurable. If \(\sigma\mapsto m_\sigma\) is log-H\"older continuous in the H\"ormander--Mihlin topology, then
\[
\sup_{k\in\Z}\left\|T_{m_{\sigma(\cdot)}}P_kf\right\|_{L^{2,\infty}(\R^n)}\lesssim\|f\|_{L^2(\R^n)}, \qquad \sup_{k\in\Z}\left\|T_{m_{\sigma(\cdot)}}P_kf\right\|_{L^p(\R^n)}\lesssim\|f\|_{L^p(\R^n)},\quad 2<p<\infty .
\]
\end{theorem}
%%%%%%%%%%%%%%%%%%%%%%%%%%%%%% THEOREM THEOREM THEOREM
Three consequences follow. The first is the sharp finite-cardinality estimate
\begin{equation}\label{eq:codim1-card}
\sup_{\{m_\sigma\}} \sup_{\substack{\Sigma\subset\Gr(n-1,n)\\ \#\Sigma\leq N}} \left\|T_{m,\Sigma}\right\|_{L^2(\R^n)} \simeq_n\log N,
\end{equation}
Beyond the band-limited setting the picture is incomplete. Fl\'orez-Amatriain \cite{Florez} has obtained \(L^p\) bounds for \(p>3/2\) when the subspace depends measurably on \(n-1\) horizontal variables, under the natural non-horizontal assumption on \(\sigma(\cdot)\). The range extends to \(p>1\) if $f$ is frequency supported in a horizontal band.

%%%%%%%%%%%%%%%%%%%%%%%%%%%%%% SECTION SECTION SECTION
\section{Algebraic almost orthogonality in higher dimensions}\label{sec:algebraic-AO} In higher dimensions, where \(\mathbb S^{n-1}\) carries no ordering, polynomial partitioning supplies a separating set \(\Omega\) that both simplifies the geometry and divides \(V\) into uniformly smaller pieces.\

We use the following form of polynomial partitioning of a set of directions, which is \cite{DPPAJM}*{Proposition 2.10}. The case \(n=3\) is stated as \cite{DPPsubspace}*{Proposition 5.11}. Given \(p\in\R[x_1,\ldots,x_n]\) we write \(Z(p)\coloneqq \{x\in\R^n:\,p(x)=0\}\), and we denote by \(\mathcal C(p)\) the set of connected components, or \emph{cells}, of \(\mathbb S^{n-1}\setminus Z(p)\).

%%%%%%%%%%%%%%%%%%%%%%%%%%%%%% PROPOSITION PROPOSITION PROPOSITION
\begin{proposition}\label{prop:polypart}
Let \(V\subset\mathbb S^{n-1}\) be finite with \(\#V\leq N^{n-1}\), let \(\delta>0\), and let \(D\geq2\) be an integer. There exists \(p\in\R[x_1,\ldots,x_n]\) of degree \(\lesssim D\) and a partition \(V=V_\circ\cup V_\times\) such that
\[
\sup_{v\in V_\times}\dist\left(v,Z(p)\right)<\delta, \qquad \#\mathcal C(p)\lesssim D^{n-1}, \qquad \sup_{C\in\mathcal C(p)}\#\left(V_\circ\cap C\right)\lesssim\left(\frac ND\right)^{n-1}.
\]
\end{proposition}
%%%%%%%%%%%%%%%%%%%%%%%%%%%%%% PROPOSITION PROPOSITION PROPOSITION

The partitioning input rests on ideas and results of Barone--Basu \cite{BaroneBasu}, Basu--Sombra \cite{BaSo}, Guth \cites{G1,Guth2}, Guth--Katz \cite{GK}, Matou\v{s}ek--Pat\'akov\'a \cite{MP} and Zahl \cites{Zahl,Zahl2}; see also \cite{sheffer} for a textbook account. The directions in \(V_\times\) are near \(Z(p)\) rather than on it. However, as the dependence on the parameter $\delta$ is qualitative, we may replace them by a direction on \(Z(p)\) at no cost, by \eqref{eq:switch} and a qualitative Fourier support condition. We therefore assume \(V_\times\subset Z(p)\) throughout.

%%%%%%%%%%%%%%%%%%%%%%%%%%%%%% SUBSECTION SUBSECTION SUBSECTION
The separating set is now \(\Omega\coloneqq \mathbb S^{n-1}\cap Z(p)\), an algebraic subvariety of \(\mathbb S^{n-1}\) of dimension \(n-2\) and controlled degree, and the cells \(C\in\mathcal C(p)\) play the role of the arcs \(\Omega_j\) of \S\ref{sec:AO-plane}.

%%%%%%%%%%%%%%%%%%%%%%%%%%%%%% PROPOSITION PROPOSITION PROPOSITION
\begin{theorem}[Algebraic almost orthogonality]\label{thm:AO-algebraic}
Let \(p\in\R[x_1,\ldots,x_n]\) have degree \(D\) and set \(\Omega\coloneqq \mathbb S^{n-1}\cap Z(p)\). Then, for every finite \(V\subset\mathbb S^{n-1}\) and every \(S\subset(0,\infty)\),
\begin{equation}\label{eq:AO-algebraic}
\left\|A_{V,S}\right\|_{L^2(\R^n)}\lesssim \sup_{\substack{U\subset\Omega\\ \#U\leq\#V}}\left\|A_{U,S}\right\|_{L^2(\R^n)}+D^{\frac{n-2}{2}}\sup_{C\in\mathcal C(p)}\left\|A_{V\cap C,S}\right\|_{L^2(\R^n)} .
\end{equation}
\end{theorem}
%%%%%%%%%%%%%%%%%%%%%%%%%%%%%% THEOREM THEOREM THEOREM

This is \cite{DPPsubspace}*{Theorem E}, and the case \(n=2\) recovers Theorem \ref{thm:AO-plane} up to the value of the constants, since \(D^{(n-2)/2}=1\) there. The gain is that the wall term involves only directions constrained to a variety of dimension \(n-2\) and degree \(O(D)\), an object of strictly lower complexity than \(\mathbb S^{n-1}\).

Fix \(C\in\mathcal C(p)\) and \(v\in V\cap C\), let \(u(v)\in\Omega\) be a closest point of \(\Omega\) to \(v\) and put
\begin{equation}\label{eq:delta-cell}
\delta(v)\coloneqq \tfrac14\left|v-u(v)\right|=\tfrac14\dist\left(v,\Omega\right),\qquad \Gamma_C\coloneqq \bigcup_{v\in V\cap C}\Gamma_{v,\delta(v)},
\end{equation}
denoting as usual by \(\Gamma_C\) also the rough Fourier restriction to \(\Gamma_C\). We also set \(U(C)\coloneqq \{u(v):\,v\in V\cap C\}\) and \(U\coloneqq \bigcup_{C}U(C)\subset\Omega\), so that \(\#U\leq\#V\). The factor \(\tfrac14\) in \eqref{eq:delta-cell} guarantees that the spherical ball of radius \(\delta(v)\) about \(v\) does not meet \(\Omega\), hence lies in the cell \(C\), which is what the next lemma exploits. Below, \(H_\xi\coloneqq \mathbb S^{n-1}\cap\xi^\perp\).

%%%%%%%%%%%%%%%%%%%%%%%%%%%%%% LEMMA LEMMA LEMMA
\begin{lemma}\label{lem:cell-meets}
Let \(C\in\mathcal C(p)\) and \(\xi\in\Gamma_C\). Then \(C\cap H_\xi\neq\varnothing\).
\end{lemma}
%%%%%%%%%%%%%%%%%%%%%%%%%%%%%% LEMMA LEMMA LEMMA

%%%%%%%%%%%%%%%%%%%%%%%%%%%%%% PROOF PROOF PROOF
\begin{proof} As \(\Gamma_C\) is a union of cones we may take \(|\xi|=1\). By \eqref{eq:delta-cell} there is \(v\in V\cap C\) with \(\xi\in\Gamma_{v,\delta(v)}\), that is \(|v\cdot\xi|<2^{-2}\delta(v)\) by \eqref{eq:plate}. The orthogonal projection of \(v\) onto \(\xi^\perp\), normalized, is a point
\[
\tau\coloneqq \frac{v-(v\cdot\xi)\xi}{\left|v-(v\cdot\xi)\xi\right|}\in H_\xi, \qquad \left|v-\tau\right|\lesssim \left|v\cdot\xi\right|<\delta(v),
\]
so that \(\tau\) lies in the spherical ball of radius \(\delta(v)\) about \(v\), hence in \(C\).
\end{proof}
%%%%%%%%%%%%%%%%%%%%%%%%%%%%%% PROOF PROOF PROOF

Lemma \ref{lem:cell-meets} is the substitute for the disjointness of the cones \(\Gamma_j\) in the plane, and it is what converts an overlap count in the frequency variable into a count of connected components on an algebraic hypersurface of one dimension less.

%%%%%%%%%%%%%%%%%%%%%%%%%%%%%% LEMMA LEMMA LEMMA
\begin{lemma}\label{lem:conical-square} There holds
\begin{equation}\label{eq:conical-square}
 \sum_{C\in\mathcal C(p)}\left\|\Gamma_Cf\right\|_2^2\lesssim D^{n-2}\|f\|_2^2 .
\end{equation}
\end{lemma}
%%%%%%%%%%%%%%%%%%%%%%%%%%%%%% LEMMA LEMMA LEMMA

%%%%%%%%%%%%%%%%%%%%%%%%%%%%%% PROOF PROOF PROOF
\begin{proof} By Plancherel it suffices to show that \(\#\{C\in\mathcal C(p):\,\xi\in\Gamma_C\}\lesssim D^{n-2}\) for almost every \(\xi\in\mathbb S^{n-1}\). Fix such a \(\xi\).  If \(\xi\in\Gamma_C\), then \(C\cap H_\xi\neq\varnothing\) by
Lemma \ref{lem:cell-meets}. Each connected component of \(H_\xi\setminus Z(p)\) is contained in a unique cell of \(\mathbb S^{n-1}\setminus Z(p)\). Hence, choosing one such component contained in \(C\cap H_\xi\) for every cell \(C\) meeting \(H_\xi\) gives an injection
\[
\{C\in\mathcal C(p):C\cap H_\xi\neq\varnothing\} \hookrightarrow \pi_0(H_\xi\setminus Z(p)).
\] 
It therefore suffices to know that \(H_\xi\setminus Z(p)\) has \(\lesssim D^{n-2}\)connected components, which is the Barone--Basu bound applied on the \((n-2)\)-dimensional sphere \(H_\xi\); see \cites{BaroneBasu,sheffer}.
\end{proof}
%%%%%%%%%%%%%%%%%%%%%%%%%%%%%% PROOF PROOF PROOF

%%%%%%%%%%%%%%%%%%%%%%%%%%%%%% PROOF PROOF PROOF
\begin{proof}[Proof of Theorem \ref{thm:AO-algebraic}] The argument is that of Theorem \ref{thm:AO-plane}, run with the cells \(C\in\mathcal C(p)\) in place of the arcs \(\Omega_j\), the cones \(\Gamma_C\) of \eqref{eq:delta-cell} in place of the \(\Gamma_j\), and \(u(v)\) in place of \(\omega(v)\); the inclusion \(\Gamma_{v,\delta(v)}\subset\Gamma_C\) that replaces \eqref{eq:AO-cone-incl} is immediate from \eqref{eq:delta-cell}. Three changes are needed.

First, the cones \(\Gamma_C\) are not pairwise disjoint, so the orthogonality \eqref{eq:AO-orth} is unavailable and \eqref{eq:conical-square} takes its place, leading to the factor \(D^{\frac{n-2}2}\) in \eqref{eq:AO-algebraic}. Second, the high-frequency piece has to be kept inside the cell operator. The multiplier of \(A^{>\delta}_{v,r}\) is that of \(A_{v,r}\) times the radial cutoff of \eqref{eq:lowhigh}, which is the Fourier transform of an \(L^1\)-normalized approximate identity at scale \(r\delta\), so that \(|A^{>\delta}_{v,r}g|\lesssim\M[A_{v,r}g]\) for every \(g\), and therefore \(|A^{<\delta(v)}_{v,r}f|=|A^{<\delta(v)}_{v,r}\Gamma_Cf|\lesssim\M[A_{v,r}\Gamma_Cf]\). Were \(A^{>\delta}_{v,r}\Gamma_Cf\) sent to the wall term instead, as it is in the plane, the factor \(D^{\frac{n-2}2}\) would multiply that term as well and the recursion below would not close.
Third, and consequently, in the analogue of \eqref{eq:AO-three} only two terms survive, and the square function is applied to the cellular one alone; the wall term is estimated directly by \(\|A_{U,S}\|_{L^2(\R^n)}\) using \eqref{eq:scales-equiv}, and \(U\subset\Omega\) has \(\#U\leq\#V\). Everything else, including the application of \eqref{eq:switch} to \(f\), is unchanged.
\end{proof}
%%%%%%%%%%%%%%%%%%%%%%%%%%%%%% PROOF PROOF PROOF

Theorem \ref{thm:AO-algebraic} separates the directions on the wall \(\Omega\) from those in the cells, which may be split again by Proposition \ref{prop:polypart}. Iterating the cell step and optimizing in \(D\) closes a recursion in the cardinality at fixed ambient dimension, once the wall directions are bounded separately. For \(n\geq4\) this would require a second descent step, and none is available; see \ref{op:descent}. In \(\R^3\), instead, the wall is an algebraic curve, the terminal case of the descent, where at single scale the answer is logarithmic rather than polynomial.

%%%%%%%%%%%%%%%%%%%%%%%%%%%%%% THEOREM THEOREM THEOREM
\begin{theorem}[\cite{DPPAJM}*{Theorem D}]\label{thm:curve-single}
Let \(Z\subset\mathbb S^2\) be a smooth real algebraic curve of degree \(D\), given by \(Z=\{P=0\}\cap \mathbb S^2\), with \(P\) transverse to \(\mathbb S^2\) along $Z$. Then
\begin{equation}\label{eq:curve-single}
 \sup_{\substack{V\subset Z\\ \#V\leq N}}\left\|A_{V,\{1\}}\right\|_{L^2(\R^3)}\lesssim D^{\frac12}\left(\log N\right)^{\frac32}.
\end{equation}
\end{theorem}
%%%%%%%%%%%%%%%%%%%%%%%%%%%%%% THEOREM THEOREM THEOREM

Estimate \eqref{eq:curve-single}  closes the recursion on Theorem \ref{thm:R3}. The transversality hypothesis is technical and can be removed up to further splitting. The corresponding multiscale problem is geometric rather than combinatorial. Already at a quadratic tangency, a band of width \(\delta\) meets the curve in an arc of length \(\simeq\sqrt\delta\) rather than \(\simeq\delta\), so a partition of the curve into arcs no longer localizes the relevant frequency cones, and the overlap count behind \eqref{eq:conical-square} is unavailable. Sharp multiscale bounds for directions on an arbitrary algebraic curve are the subject of work in progress of the authors; see also \ref{op:curves-n}.

%%%%%%%%%%%%%%%%%%%%%%%%%%%%%% SECTION SECTION SECTION
\section{Open problems}\label{sec:open} We collect the preceding questions by the mechanism needed to resolve them. Several are long-standing; a few, as far as we know, are new in this form.

%%%%%%%%%%%%%%%%%%%%%%%%%%%%%% SUBSECTION SUBSECTION SUBSECTION
\subsection{The plane} The planar theory is the most developed, which makes the remaining gaps sharper in outline.

\begin{enumerate}[label=\textup{(\roman*)},series=openproblems,leftmargin=*,labelsep=0.4em]
\item\label{op:weak}  Decide whether the weak type estimate
\begin{equation}\label{eq:weak-conj}
 \left\|T_{m,V}\right\|_{L^2(\R^2)\to L^{2,\infty}(\R^2)}\lesssim \sqrt{\log N}
\end{equation}
holds uniformly over all \(V\subset\mathbb S^1\) with \(\#V=N\) and over a fixed H\"ormander--Mihlin class. This would be  the exact singular integral counterpart of Katz's weak type improvement for \(\M_V\), and by \S\ref{sec:general} it would also settle the conjectured \(L^p\) bound \eqref{eq:DDP-Lp1}. Estimate \eqref{eq:weak-conj} is open even for the Hilbert transform, outside of  structured classes of \S\ref{sec:equi}, up to higher logarithms.
\item\label{op:intermediate} Between the two extremes \(\sqrt{\log N}\), attained on finite-order lacunary sets, and \(\log N\), attained on arbitrary sets, lies an entire range that is at present uncharted. Characterize the sets \(V\subset\mathbb S^1\) for which
\[
 \sup_{\substack{V_N\subset V\\ \#V_N\leq N}}
 \left\|H_{V_N}\right\|_{L^2(\R^2)}=o\left(\log N\right),
 \qquad N\to\infty.
\]
Unlike the maximal average, \(H_V\) is unbounded for \emph{every} infinite \(V\) \cites{Kara,LMP}, so no dichotomy of lacunary type can hold here and the question is genuinely quantitative.

\item\label{op:polygon} Determine the sharp growth in \(N\) of the strong or restricted type norm of the Fourier restriction to a regular \(N\)-gon on \(L^4(\R^2)\). The best known upper bound may be found in \cite{ADPHPR}. No matching lower bound is known.

\item\label{op:maxpolygon} Pass from a single dilate to the supremum over dilates. With \(T_{\mathcal P,\lambda}\) as in \eqref{eq:polygon-dilate}, prove that the maximal polygon multiplier
\[
 T_{N,\star}f\coloneqq \sup_{\lambda>0}\left|T_{\mathcal P_N,\lambda}f\right|
\]
satisfies \(L^p(\R^2)\) bounds depending logarithmically on \(N\), for \(\tfrac43\leq p\leq4\); the estimates presently available are of power type. Once the Besicovitch obstruction is removed altogether, even a bound independent of the number of sides is missing: prove that if \(\mathcal P\) is a convex polygon whose slopes form a lacunary set, then \(\sup_{\lambda>0}\left|T_{\mathcal P,\lambda}f\right|\) is bounded on \(L^p(\R^2)\) for \(1<p<\infty\), with a bound independent of the number of sides of \(\mathcal P\). This last question separates the two difficulties this survey has been concerned with. 
\end{enumerate}

%%%%%%%%%%%%%%%%%%%%%%%%%%%%%% SUBSECTION SUBSECTION SUBSECTION
\subsection{Vector fields} Here the obstruction is not counting but the absence of any finite parameter measuring the complexity of the direction field.

\begin{enumerate}[label=\textup{(\roman*)},resume*=openproblems,leftmargin=*,labelsep=0.4em]

\item\label{op:zygmund} Zygmund's conjecture. Characterize the Lipschitz vector fields \(v:\R^2\to\mathbb S^1\) along which \(L^2(\R^2)\) functions are differentiated, and Stein's singular integral counterpart, the weak \((2,2)\) bound for the truncated directional Hilbert transform along a Lipschitz field. Both remain open beyond the band-limited setting of \S\ref{sec:vf}, and beyond the one-variable case of \cite{BaTh}. 

\item\label{op:lipkakeya} Prove the conjecture of \cite{LaceyLiMem} on the Lipschitz Kakeya maximal function \eqref{eq:lipschitz-kakeya}: that \(\mathrm M_{v,\delta}\) maps \(L^p\) to \(L^{p,\infty}\) with norm \(\lesssim\delta^{-N}\) for some \(p<2\) and some finite \(N\), uniformly over Lipschitz fields \(v\). By \S\ref{sec:vf} this would give the single annulus \(L^2\) bound for every Lipschitz field, and only the integrability exponent in \eqref{eq:LK-weak2} stands between the two statements. Together with \ref{op:vv} it is one of two independent routes to \ref{op:zygmund}.

\item\label{op:vv} Prove vector-valued extensions of the Lacey--Li single-annulus theorem \cite{LaceyLi}, for instance weak-type $(2,2)$ or strong-type $(p,p)$ bounds for
\begin{equation}\label{eq:vv-square}
 f\mapsto\left(\sum_k\left|H_{v(\cdot)}P_kf\right|^2\right)^{\frac12},
\end{equation}
where \(\{P_k\}\) is a Littlewood--Paley decomposition of \(\R^2\) and \(v\) is a measurable vector field. By \S\ref{sec:vf} this implies Stein's conjecture for Lipschitz fields.

\end{enumerate}

%%%%%%%%%%%%%%%%%%%%%%%%%%%%%% SUBSECTION SUBSECTION SUBSECTION
\subsection{Higher dimensions and the Grassmannian}

\begin{enumerate}[label=\textup{(\roman*)},resume*=openproblems,leftmargin=*,labelsep=0.4em]

\item\label{op:nikodym} For arbitrary finite \(V\subset\mathbb S^{n-1}\), \(n\geq3\), interpolating Theorem~\ref{thm:varieties} at \(m=n-1\) with the trivial \(L^\infty\) bound gives \(\|\M_{V,1}\|_{L^p(\R^n)}\lesssim_\eps(\#V)^{\frac{n-2}{p(n-1)}+\eps}\) for \(p\geq2\), against the lower bound \((\#V)^{\frac{n-p}{p(n-1)}}\) of \eqref{eq:radial-subcritical}. Improve this exponent for some \(2<p<n\); any \(p>2\) would be nontrivial. For the tube maximal function \(\mathcal N_\delta\), Wolff \cite{Wolff95} obtains the corresponding sharp exponent at \(p=\frac{n+2}{2}\), up to \(\eps\)-loss, but the comparison \eqref{eq:nikodym-comparison} runs only from \(\M_{V,1}\) to \(\mathcal N_\delta\). Closing the gap at \(p=n\) is \eqref{eq:nikodym-max}, which contains the maximal Nikodym and maximal Kakeya conjectures.

\item\label{op:conjG} Establish Conjecture \ref{conj:nikodym} for \(1<d<n-1\) outside \(L^2\). The case \(d=n-1\), where the critical exponent is \(p_\circ=2\), is known; for intermediate \(d\) the critical space \(L^{n-d+1}(\R^n)\) is out of reach, and even the subcritical range is only understood at \(p=2\).
\end{enumerate}

%%%%%%%%%%%%%%%%%%%%%%%%%%%%%% SUBSECTION SUBSECTION SUBSECTION
\subsection{Algebraic structure} The last group concerns the reach of the polynomial method in this setting, and is where we expect the next progress to come from.

\begin{enumerate}[label=\textup{(\roman*)},resume*=openproblems,leftmargin=1.5em,labelsep=0.4em]

\item\label{op:maximal-alg} Develop sharp almost-orthogonality estimates for maximal averages, and not only for band multipliers, along algebraic varieties in higher dimensions. For band multipliers a square function estimate of C\'ordoba type is available and, coupled with Theorem \ref{thm:AO-algebraic}, yields sharp bounds; for maximal averages the supremum over scales obstructs that estimate, and no substitute is currently known.

\item\label{op:descent} Find a form of Theorem \ref{thm:AO-algebraic} in which \(\mathbb S^{n-1}\) is replaced by a variety of arbitrary dimension. The proof of Lemma \ref{lem:cell-meets} normalizes the projection of \(v\) onto \(\xi^\perp\), returning a point of \(\mathbb S^{n-1}\) rather than of the carrying variety; the component-counting argument on \(H_\xi\setminus Z(p)\) is then lost, and the principle descends one step only. Such a statement would replace the induction on dimension performed by hand in \cite{DPPAJM}. The same question on \(\Gr(d,n)\), \(d>1\), would govern subspaces constrained to a subvariety of the Grassmannian, already in the first new case \(\Gr(2,4)\).

\item\label{op:curves-n} Find the sharp multiscale bound for maximal averages along directions lying on an algebraic curve, first in \(\mathbb S^2\) and then in \(\mathbb S^{n-1}\), and extend it to varieties of intermediate dimension. 

\end{enumerate}

\bibliographystyle{plain}

\begin{bibdiv}
\begin{biblist}

\bib{ADPHPR}{article}{
      author={Accomazzo, Natalia},
      author={Di~Plinio, Francesco},
      author={Hagelstein, Paul},
      author={Parissis, Ioannis},
      author={Roncal, Luz},
       title={Directional square functions},
        date={2023},
        ISSN={2157-5045,1948-206X},
     journal={Anal. PDE},
      volume={16},
      number={7},
       pages={1651\ndash 1699},
         url={https://doi.org/10.2140/apde.2023.16.1651},
      review={\MR{4644118}},
}

\bib{ADPP}{article}{
      author={Accomazzo, Natalia},
      author={Di~Plinio, Francesco},
      author={Parissis, Ioannis},
       title={Singular integrals along lacunary directions in {$\mathbb
  {R}^n$}},
        date={2021},
        ISSN={0001-8708,1090-2082},
     journal={Adv. Math.},
      volume={380},
       pages={Paper No. 107580, 21},
         url={https://doi.org/10.1016/j.aim.2021.107580},
      review={\MR{4200473}},
}

\bib{Alf_AO}{article}{
      author={Alfonseca, Angeles},
       title={Strong type inequalities and an almost-orthogonality principle
  for families of maximal operators along directions in {$\mathbb R^2$}},
        date={2003},
        ISSN={0024-6107,1469-7750},
     journal={J. London Math. Soc. (2)},
      volume={67},
      number={1},
       pages={208\ndash 218},
         url={https://doi.org/10.1112/S0024610702003915},
      review={\MR{1942421}},
}

\bib{ASVmounthol}{incollection}{
      author={Alfonseca, Angeles},
      author={Soria, Fernando},
      author={Vargas, Ana},
       title={An almost-orthogonality principle in {$L^2$} for directional
  maximal functions},
        date={2003},
   booktitle={Harmonic analysis at {M}ount {H}olyoke ({S}outh {H}adley, {MA},
  2001)},
      series={Contemp. Math.},
      volume={320},
   publisher={Amer. Math. Soc., Providence, RI},
       pages={1\ndash 7},
         url={https://doi.org/10.1090/conm/320/05594},
      review={\MR{1979927}},
}

\bib{ASV}{article}{
      author={Alfonseca, Angeles},
      author={Soria, Fernando},
      author={Vargas, Ana},
       title={A remark on maximal operators along directions in {$\mathbb
  R^2$}},
        date={2003},
        ISSN={1073-2780},
     journal={Math. Res. Lett.},
      volume={10},
      number={1},
       pages={41\ndash 49},
         url={https://doi.org/10.4310/MRL.2003.v10.n1.a5},
      review={\MR{1960122}},
}

\bib{BDPPR}{article}{
      author={Bakas, Odysseas},
      author={Di~Plinio, Francesco},
      author={Parissis, Ioannis},
      author={Roncal, Luz},
       title={Singular integrals along variable codimension one subspaces},
        date={2024},
        ISSN={2769-8505},
     journal={Ars Inven. Anal.},
       pages={Paper No. 2, 53},
         url={https://doi.org/10.15781/md28-ws10},
      review={\MR{4723191}},
}

\bib{BaroneBasu}{article}{
      author={Barone, Sal},
      author={Basu, Saugata},
       title={Refined bounds on the number of connected components of sign
  conditions on a variety},
        date={2012},
        ISSN={0179-5376},
     journal={Discrete Comput. Geom.},
      volume={47},
      number={3},
       pages={577\ndash 597},
      review={\MR{2891249}},
}

\bib{BaSo}{article}{
      author={Basu, Saugata},
      author={Sombra, Mart{\'i}n},
       title={Polynomial partitioning on varieties of codimension two and
  point-hypersurface incidences in four dimensions},
        date={2016},
        ISSN={0179-5376},
     journal={Discrete Comput. Geom.},
      volume={55},
      number={1},
       pages={158\ndash 184},
         url={https://doi.org/10.1007/s00454-015-9736-4},
      review={\MR{3439263}},
}

\bib{Batduke}{article}{
      author={Bateman, Michael},
       title={Kakeya sets and directional maximal operators in the plane},
        date={2009},
        ISSN={0012-7094,1547-7398},
     journal={Duke Math. J.},
      volume={147},
      number={1},
       pages={55\ndash 77},
         url={https://doi.org/10.1215/00127094-2009-006},
      review={\MR{2494456}},
}

\bib{BatRMI}{article}{
      author={Bateman, Michael},
       title={Single annulus {$L^p$} estimates for {H}ilbert transforms along
  vector fields},
        date={2013},
        ISSN={0213-2230},
     journal={Rev. Mat. Iberoam.},
      volume={29},
      number={3},
       pages={1021\ndash 1069},
         url={https://doi.org/10.4171/RMI/748},
      review={\MR{3090145}},
}

\bib{BKcantor}{article}{
      author={Bateman, Michael},
      author={Katz, Nets~Hawk},
       title={Kakeya sets in {C}antor directions},
        date={2008},
     journal={Math. Res. Lett.},
      volume={15},
      number={1},
       pages={73\ndash 81},
      review={\MR{2367175}},
}

\bib{BaTh}{article}{
      author={Bateman, Michael},
      author={Thiele, Christoph},
       title={{$L^p$} estimates for the {H}ilbert transforms along a
  one-variable vector field},
        date={2013},
        ISSN={2157-5045},
     journal={Anal. PDE},
      volume={6},
      number={7},
       pages={1577\ndash 1600},
         url={https://doi.org/10.2140/apde.2013.6.1577},
      review={\MR{3148061}},
}

\bib{Besi}{article}{
      author={Besicovitch, A.~S.},
       title={On {K}akeya's problem and a similar one},
        date={1928},
        ISSN={0025-5874,1432-1823},
     journal={Math. Z.},
      volume={27},
      number={1},
       pages={312\ndash 320},
         url={https://doi.org/10.1007/BF01171101},
      review={\MR{1544912}},
}

\bib{PeBi}{book}{
      author={Bishop, Christopher~J.},
      author={Peres, Yuval},
       title={Fractals in probability and analysis},
      series={Cambridge Studies in Advanced Mathematics},
   publisher={Cambridge University Press, Cambridge},
        date={2017},
      volume={162},
        ISBN={978-1-107-13411-9},
         url={https://doi.org/10.1017/9781316460238},
      review={\MR{3616046}},
}

\bib{Bourgain89}{incollection}{
      author={Bourgain, J.},
       title={A remark on the maximal function associated to an analytic vector
  field},
        date={1989},
   booktitle={Analysis at {U}rbana, {V}ol. {I} ({U}rbana, {IL}, 1986--1987)},
      series={London Math. Soc. Lecture Note Ser.},
      volume={137},
   publisher={Cambridge Univ. Press, Cambridge},
       pages={111\ndash 132},
      review={\MR{1009171}},
}

\bib{Bourgain91}{article}{
      author={Bourgain, J.},
       title={Besicovitch type maximal operators and applications to {F}ourier
  analysis},
        date={1991},
     journal={Geom. Funct. Anal.},
      volume={1},
      number={2},
       pages={147\ndash 187},
      review={\MR{1097257}},
}

\bib{Car}{article}{
      author={Carbery, Anthony},
       title={Covering lemmas revisited},
        date={1988},
        ISSN={0013-0915,1464-3839},
     journal={Proc. Edinburgh Math. Soc. (2)},
      volume={31},
      number={1},
       pages={145\ndash 150},
         url={https://doi.org/10.1017/S0013091500006647},
      review={\MR{930022}},
}

\bib{CaHeSo}{article}{
      author={Carbery, Anthony},
      author={Hern\'andez, Eugenio},
      author={Soria, Fernando},
       title={The behaviour on radial functions of maximal operators along
  arbitrary directions and the {K}akeya maximal operator},
        date={1989},
        ISSN={0040-8735,2186-585X},
     journal={Tohoku Math. J. (2)},
      volume={41},
      number={4},
       pages={647\ndash 656},
         url={https://doi.org/10.2748/tmj/1178227735},
      review={\MR{1025330}},
}

\bib{CaHeSoer}{article}{
      author={Carbery, Anthony},
      author={Hern\'andez, Eugenio},
      author={Soria, Fernando},
       title={Erratum: ``{T}he behaviour on radial functions of maximal
  operators along arbitrary directions and the {K}akeya maximal operator''},
        date={1990},
        ISSN={0040-8735,2186-585X},
     journal={Tohoku Math. J. (2)},
      volume={42},
      number={2},
       pages={287},
         url={https://doi.org/10.2748/tmj/1178227658},
      review={\MR{1053953}},
}

\bib{CWW}{article}{
      author={Chang, S.-Y.\~A.},
      author={Wilson, J.~M.},
      author={Wolff, T.~H.},
       title={Some weighted norm inequalities concerning the {S}chr\"odinger
  operators},
        date={1985},
        ISSN={0010-2571,1420-8946},
     journal={Comment. Math. Helv.},
      volume={60},
      number={2},
       pages={217\ndash 246},
         url={https://doi.org/10.1007/BF02567411},
      review={\MR{800004}},
}

\bib{CDRdF}{article}{
      author={Christ, Michael},
      author={Duoandikoetxea, Javier},
      author={Rubio~de Francia, Jos\'e~L.},
       title={Maximal operators related to the {R}adon transform and the
  {C}alder\'on-{Z}ygmund method of rotations},
        date={1986},
     journal={Duke Math. J.},
      volume={53},
      number={1},
       pages={189\ndash 209},
      review={\MR{835805}},
}

\bib{CGHS}{article}{
      author={Christ, Michael},
      author={Grafakos, Loukas},
      author={Honz\'ik, Petr},
      author={Seeger, Andreas},
       title={Maximal functions associated with {F}ourier multipliers of
  {M}ikhlin-{H}\"ormander type},
        date={2005},
        ISSN={0025-5874,1432-1823},
     journal={Math. Z.},
      volume={249},
      number={1},
       pages={223\ndash 240},
         url={https://doi.org/10.1007/s00209-004-0698-0},
      review={\MR{2106977}},
}

\bib{CNSW}{article}{
      author={Christ, Michael},
      author={Nagel, Alexander},
      author={Stein, Elias~M.},
      author={Wainger, Stephen},
       title={Singular and maximal {R}adon transforms: analysis and geometry},
        date={1999},
     journal={Ann. of Math. (2)},
      volume={150},
      number={2},
       pages={489\ndash 577},
      review={\MR{1726701}},
}

\bib{CordPoly}{article}{
      author={C\'ordoba, A.},
       title={The multiplier problem for the polygon},
        date={1977},
     journal={Ann. of Math. (2)},
      volume={105},
      number={3},
       pages={581\ndash 588},
      review={\MR{0438022}},
}

\bib{CoFeflac}{article}{
      author={C\'ordoba, A.},
      author={Fefferman, R.},
       title={On differentiation of integrals},
        date={1977},
        ISSN={0027-8424},
     journal={Proc. Nat. Acad. Sci. U.S.A.},
      volume={74},
      number={6},
       pages={2211\ndash 2213},
         url={https://doi.org/10.1073/pnas.74.6.2211},
      review={\MR{476977}},
}

\bib{CoFemult}{article}{
      author={C\'ordoba, A.},
      author={Fefferman, R.},
       title={On the equivalence between the boundedness of certain classes of
  maximal and multiplier operators in {F}ourier analysis},
        date={1977},
        ISSN={0027-8424},
     journal={Proc. Nat. Acad. Sci. U.S.A.},
      volume={74},
      number={2},
       pages={423\ndash 425},
      review={\MR{0433117}},
}

\bib{CordAJM}{article}{
      author={Cordoba, Antonio},
       title={The {K}akeya maximal function and the spherical summation
  multipliers},
        date={1977},
        ISSN={0002-9327,1080-6377},
     journal={Amer. J. Math.},
      volume={99},
      number={1},
       pages={1\ndash 22},
         url={https://doi.org/10.2307/2374006},
      review={\MR{447949}},
}

\bib{CorGFA}{article}{
      author={C\'ordoba, Antonio},
       title={Geometric {F}ourier analysis},
        date={1982},
        ISSN={0373-0956},
     journal={Ann. Inst. Fourier (Grenoble)},
      volume={32},
      number={3},
       pages={vii, 215\ndash 226},
         url={http://www.numdam.org/item?id=AIF_1982__32_3_215_0},
      review={\MR{688027}},
}

\bib{ciprianpams}{article}{
      author={Demeter, Ciprian},
       title={Singular integrals along {$N$} directions in {$\mathbb R^2$}},
        date={2010},
        ISSN={0002-9939,1088-6826},
     journal={Proc. Amer. Math. Soc.},
      volume={138},
      number={12},
       pages={4433\ndash 4442},
         url={https://doi.org/10.1090/S0002-9939-2010-10442-2},
      review={\MR{2680067}},
}

\bib{cipriangreedy}{article}{
      author={Demeter, Ciprian},
       title={{$L^2$} bounds for a {K}akeya-type maximal operator in {$\mathbb
  R^3$}},
        date={2012},
        ISSN={0024-6093,1469-2120},
     journal={Bull. Lond. Math. Soc.},
      volume={44},
      number={4},
       pages={716\ndash 728},
         url={https://doi.org/10.1112/blms/bds004},
      review={\MR{2967239}},
}

\bib{DDP}{article}{
      author={Demeter, Ciprian},
      author={Di~Plinio, Francesco},
       title={Logarithmic {$L^p$} bounds for maximal directional singular
  integrals in the plane},
        date={2014},
        ISSN={1050-6926,1559-002X},
     journal={J. Geom. Anal.},
      volume={24},
      number={1},
       pages={375\ndash 416},
         url={https://doi.org/10.1007/s12220-012-9340-2},
      review={\MR{3145928}},
}

\bib{DPGTZK}{article}{
      author={Di~Plinio, Francesco},
      author={Guo, Shaoming},
      author={Thiele, Christoph},
      author={Zorin-Kranich, Pavel},
       title={Square functions for bi-{L}ipschitz maps and directional
  operators},
        date={2018},
        ISSN={0022-1236},
     journal={J. Funct. Anal.},
      volume={275},
      number={8},
       pages={2015\ndash 2058},
         url={https://doi.org/10.1016/j.jfa.2018.07.005},
      review={\MR{3841536}},
}

\bib{DPPIJM}{article}{
      author={Di~Plinio, Francesco},
      author={Parissis, Ioannis},
       title={A sharp estimate for the {H}ilbert transform along finite order
  lacunary sets of directions},
        date={2018},
        ISSN={0021-2172,1565-8511},
     journal={Israel J. Math.},
      volume={227},
      number={1},
       pages={189\ndash 214},
         url={https://doi.org/10.1007/s11856-018-1724-y},
      review={\MR{3846321}},
}

\bib{DPPIMRN}{article}{
      author={Di~Plinio, Francesco},
      author={Parissis, Ioannis},
       title={On the maximal directional {H}ilbert transform in three
  dimensions},
        date={2020},
        ISSN={1073-7928,1687-0247},
     journal={Int. Math. Res. Not. IMRN},
      number={14},
       pages={4324\ndash 4356},
         url={https://doi.org/10.1093/imrn/rny138},
      review={\MR{4126303}},
}

\bib{DPPAJM}{article}{
      author={Di~Plinio, Francesco},
      author={Parissis, Ioannis},
       title={Maximal directional operators along algebraic varieties},
        date={2021},
        ISSN={0002-9327,1080-6377},
     journal={Amer. J. Math.},
      volume={143},
      number={5},
       pages={1463\ndash 1503},
         url={https://doi.org/10.1353/ajm.2021.0037},
      review={\MR{4334401}},
}

\bib{DPPsubspace}{article}{
      author={Di~Plinio, Francesco},
      author={Parissis, Ioannis},
       title={Maximal subspace averages},
        date={2022},
        ISSN={0001-8708,1090-2082},
     journal={Adv. Math.},
      volume={410},
       pages={Paper No. 108749, 51},
         url={https://doi.org/10.1016/j.aim.2022.108749},
      review={\MR{4498204}},
}

\bib{DuoVargas}{article}{
      author={Duoandikoetxea, Javier},
      author={Vargas, Ana},
       title={Directional operators and radial functions on the plane},
        date={1995},
        ISSN={0004-2080,1871-2487},
     journal={Ark. Mat.},
      volume={33},
      number={2},
       pages={281\ndash 291},
         url={https://doi.org/10.1007/BF02559710},
      review={\MR{1373025}},
}

\bib{Falconer80}{article}{
      author={Falconer, K.~J.},
       title={Continuity properties of {$k$}-plane integrals and {B}esicovitch
  sets},
        date={1980},
     journal={Math. Proc. Cambridge Philos. Soc.},
      volume={87},
      number={2},
       pages={221\ndash 226},
      review={\MR{553579}},
}

\bib{Falconer86}{article}{
      author={Falconer, K.~J.},
       title={Sets with prescribed projections and {N}ikod\'ym sets},
        date={1986},
     journal={Proc. London Math. Soc. (3)},
      volume={53},
      number={1},
       pages={48\ndash 64},
      review={\MR{842156}},
}

\bib{FeffBall}{article}{
      author={Fefferman, Charles},
       title={The multiplier problem for the ball},
        date={1971},
     journal={Ann. of Math. (2)},
      volume={94},
       pages={330\ndash 336},
      review={\MR{0296602}},
}

\bib{Florez}{unpublished}{
      author={Fl\'orez-Amatriain, Mikel},
       title={{$L^p$}-estimates for singular integral operators along
  codimension one subspaces},
        date={2025},
        note={Preprint, \eprint{2502.13079}},
}

\bib{GanWu}{article}{
      author={Gan, Shengwen},
      author={Wu, Shukun},
       title={Square function estimates for conical regions},
     journal={Indiana Univ. Math. J.},
        note={\eprint{2203.12155}},
}

\bib{Graf}{book}{
      author={Grafakos, Loukas},
       title={Classical {F}ourier analysis},
     edition={Third},
      series={Graduate Texts in Mathematics},
   publisher={Springer, New York},
        date={2014},
      volume={249},
        ISBN={978-1-4939-1193-6; 978-1-4939-1194-3},
         url={https://doi.org/10.1007/978-1-4939-1194-3},
      review={\MR{3243734}},
}

\bib{GHS}{article}{
      author={Grafakos, Loukas},
      author={Honz\'ik, Petr},
      author={Seeger, Andreas},
       title={On maximal functions for {M}ikhlin-{H}\"ormander multipliers},
        date={2006},
        ISSN={0001-8708,1090-2082},
     journal={Adv. Math.},
      volume={204},
      number={2},
       pages={363\ndash 378},
         url={https://doi.org/10.1016/j.aim.2005.05.010},
      review={\MR{2249617}},
}

\bib{GuoAPDE}{article}{
      author={Guo, Shaoming},
       title={Hilbert transform along measurable vector fields constant on
  {L}ipschitz curves: {$L^2$} boundedness},
        date={2015},
     journal={Anal. PDE},
      volume={8},
      number={5},
       pages={1263\ndash 1288},
      review={\MR{3393679}},
}

\bib{GuoJGA}{article}{
      author={Guo, Shaoming},
       title={A geometric proof of {B}ourgain's {$L^2$} estimate of the maximal
  operator along analytic vector fields},
        date={2017},
        ISSN={1050-6926},
     journal={J. Geom. Anal.},
      volume={27},
      number={4},
       pages={2977\ndash 3012},
}

\bib{G1}{article}{
      author={Guth, Larry},
       title={A restriction estimate using polynomial partitioning},
        date={2016},
        ISSN={0894-0347},
     journal={J. Amer. Math. Soc.},
      volume={29},
      number={2},
       pages={371\ndash 413},
         url={https://doi.org/10.1090/jams827},
      review={\MR{3454378}},
}

\bib{Guth2}{article}{
      author={Guth, Larry},
       title={Restriction estimates using polynomial partitioning {II}},
        date={201603},
      eprint={1603.04250},
         url={https://arxiv.org/abs/1603.04250},
}

\bib{GK}{article}{
      author={Guth, Larry},
      author={Katz, Nets~Hawk},
       title={On the {E}rd{\"o}s distinct distances problem in the plane},
        date={2015},
        ISSN={0003-486X},
     journal={Ann. of Math. (2)},
      volume={181},
      number={1},
       pages={155\ndash 190},
         url={https://doi.org/10.4007/annals.2015.181.1.2},
      review={\MR{3272924}},
}

\bib{HRMSduke}{article}{
      author={Hagelstein, Paul},
      author={Radillo-Murguia, Blanca},
      author={Stokolos, Alexander},
       title={Probabilistic construction of {K}akeya-type sets in {$\mathbb
  {R}^{2}$} associated to separated sets of directions},
        date={2026},
        ISSN={0012-7094,1547-7398},
     journal={Duke Math. J.},
      volume={175},
      number={2},
       pages={185\ndash 198},
         url={https://doi.org/10.1215/00127094-2025-0025},
      review={\MR{5030464}},
}

\bib{HRMS}{article}{
      author={Hagelstein, Paul},
      author={Radillo-Murguia, Blanca},
      author={Stokolos, Alexander},
       title={Remarks on the construction of {$K_\sigma$} sets associated to
  trees not satisfying a separation condition},
        date={2026},
        ISSN={0213-2230,2235-0616},
     journal={Rev. Mat. Iberoam.},
      volume={42},
      number={2},
       pages={545\ndash 550},
         url={https://doi.org/10.4171/rmi/1581},
      review={\MR{5035525}},
}

\bib{Kara}{article}{
      author={Karagulyan, G.~A.},
       title={On unboundedness of maximal operators for directional {H}ilbert
  transforms},
        date={2007},
        ISSN={0002-9939,1088-6826},
     journal={Proc. Amer. Math. Soc.},
      volume={135},
      number={10},
       pages={3133\ndash 3141},
         url={https://doi.org/10.1090/S0002-9939-07-08731-X},
      review={\MR{2322743}},
}

\bib{Katzcantor}{article}{
      author={Katz, Nets~Hawk},
       title={A counterexample for maximal operators over a {C}antor set of
  directions},
        date={1996},
        ISSN={1073-2780},
     journal={Math. Res. Lett.},
      volume={3},
      number={4},
       pages={527\ndash 536},
         url={https://doi.org/10.4310/MRL.1996.v3.n4.a10},
      review={\MR{1406017}},
}

\bib{KatzDuke}{article}{
      author={Katz, Nets~Hawk},
       title={Maximal operators over arbitrary sets of directions},
        date={1999},
        ISSN={0012-7094},
     journal={Duke Math. J.},
      volume={97},
      number={1},
       pages={67\ndash 79},
         url={https://doi.org/10.1215/S0012-7094-99-09702-8},
      review={\MR{1681088}},
}

\bib{Katzsingle}{article}{
      author={Katz, Nets~Hawk},
       title={Remarks on maximal operators over arbitrary sets of directions},
        date={1999},
        ISSN={0024-6093,1469-2120},
     journal={Bull. London Math. Soc.},
      volume={31},
      number={6},
       pages={700\ndash 710},
         url={https://doi.org/10.1112/S0024609399005949},
      review={\MR{1711029}},
}

\bib{KimPramanik}{article}{
      author={Kim, Jongchon},
      author={Pramanik, Malabika},
       title={{$L^2$} bounds for a maximal directional {H}ilbert transform},
        date={2022},
        ISSN={2157-5045,1948-206X},
     journal={Anal. PDE},
      volume={15},
      number={3},
       pages={753\ndash 794},
         url={https://doi.org/10.2140/apde.2022.15.753},
      review={\MR{4442840}},
}

\bib{JKim}{article}{
      author={Kim, Joonil},
       title={Maximal average along variable lines},
        date={2008},
     journal={Trans. Amer. Math. Soc.},
      volume={360},
      number={10},
       pages={5417\ndash 5436},
      review={\MR{2415079}},
}

\bib{KLP}{unpublished}{
      author={Kroc, Edward},
      author={Lee, Juyoung},
      author={Pramanik, Malabika},
       title={Directional maximal operators in the plane},
        date={2026},
        note={Preprint, \eprint{2608.23871}},
}

\bib{LMP}{article}{
      author={{\L}aba, I.},
      author={Marinelli, A.},
      author={Pramanik, M.},
       title={On the maximal directional {H}ilbert transform},
        date={2019},
        ISSN={0133-3852,1588-273X},
     journal={Anal. Math.},
      volume={45},
      number={3},
       pages={535\ndash 568},
         url={https://doi.org/10.1007/s10476-019-0821-4},
      review={\MR{3995379}},
}

\bib{LaceyLiMem}{article}{
      author={Lacey, Michael},
      author={Li, Xiaochun},
       title={On a conjecture of {E}. {M}. {S}tein on the {H}ilbert transform
  on vector fields},
        date={2010},
     journal={Mem. Amer. Math. Soc.},
      volume={205},
      number={965},
       pages={viii+72},
      review={\MR{2654385}},
}

\bib{LaceyLi}{article}{
      author={Lacey, Michael~T.},
      author={Li, Xiaochun},
       title={Maximal theorems for the directional {H}ilbert transform on the
  plane},
        date={2006},
        ISSN={0002-9947,1088-6850},
     journal={Trans. Amer. Math. Soc.},
      volume={358},
      number={9},
       pages={4099\ndash 4117},
         url={https://doi.org/10.1090/S0002-9947-06-03869-4},
      review={\MR{2219012}},
}

\bib{MP}{article}{
      author={Matou{\v s}ek, Ji\v~r\'\i},
      author={Pat\'akov\'a, Zuzana},
       title={Multilevel polynomial partitions and simplified range searching},
        date={2015},
        ISSN={0179-5376},
     journal={Discrete Comput. Geom.},
      volume={54},
      number={1},
       pages={22\ndash 41},
         url={https://doi.org/10.1007/s00454-015-9701-2},
      review={\MR{3351757}},
}

\bib{Mattila}{book}{
      author={Mattila, Pertti},
       title={Fourier analysis and {H}ausdorff dimension},
      series={Cambridge Studies in Advanced Mathematics},
   publisher={Cambridge University Press, Cambridge},
        date={2015},
      volume={150},
      review={\MR{3617376}},
}

\bib{NSW}{article}{
      author={Nagel, A.},
      author={Stein, E.~M.},
      author={Wainger, S.},
       title={Differentiation in lacunary directions},
        date={1978},
        ISSN={0027-8424},
     journal={Proc. Nat. Acad. Sci. U.S.A.},
      volume={75},
      number={3},
       pages={1060\ndash 1062},
         url={https://doi.org/10.1073/pnas.75.3.1060},
      review={\MR{466470}},
}

\bib{Oberlin}{article}{
      author={Oberlin, Richard},
       title={Bounds for {K}akeya-type maximal operators associated with
  {$k$}-planes},
        date={2007},
        ISSN={1073-2780},
     journal={Math. Res. Lett.},
      volume={14},
      number={1},
       pages={87\ndash 97},
      review={\MR{2289622}},
}

\bib{OrpRoy}{unpublished}{
      author={Orponen, Tuomas},
      author={Roy, Hrit},
       title={Nikod\'ym maximal function with restricted directions},
        date={2026},
        note={Preprint, \eprint{2601.19631}},
}

\bib{PRproc}{article}{
      author={Parcet, Javier},
      author={Rogers, Keith~M.},
       title={Differentiation of integrals in higher dimensions},
        date={2013},
        ISSN={0027-8424,1091-6490},
     journal={Proc. Natl. Acad. Sci. USA},
      volume={110},
      number={13},
       pages={4941\ndash 4944},
         url={https://doi.org/10.1073/pnas.1218928110},
      review={\MR{3047650}},
}

\bib{PR}{article}{
      author={Parcet, Javier},
      author={Rogers, Keith~M.},
       title={Directional maximal operators and lacunarity in higher
  dimensions},
        date={2015},
        ISSN={0002-9327,1080-6377},
     journal={Amer. J. Math.},
      volume={137},
      number={6},
       pages={1535\ndash 1557},
      review={\MR{3432267}},
}

\bib{Schoen}{incollection}{
      author={Schoenberg, I.~J.},
       title={On the {B}esicovitch-{P}erron solution of the {K}akeya problem},
        date={1962},
   booktitle={Studies in mathematical analysis and related topics},
      series={Stanford Studies in Mathematics and Statistics},
      volume={IV},
   publisher={Stanford Univ. Press, Stanford, CA},
       pages={359\ndash 363},
      review={\MR{146698}},
}

\bib{sheffer}{book}{
      author={Sheffer, Adam},
       title={Polynomial methods and incidence theory},
      series={Cambridge Studies in Advanced Mathematics},
   publisher={Cambridge University Press, Cambridge},
        date={2022},
      volume={197},
        ISBN={978-1-108-83249-6},
         url={https://doi.org/10.1017/9781108959988},
      review={\MR{4394303}},
}

\bib{SjSj}{article}{
      author={Sj\"ogren, P.},
      author={Sj\"olin, P.},
       title={Littlewood-{P}aley decompositions and {F}ourier multipliers with
  singularities on certain sets},
        date={1981},
        ISSN={0373-0956,1777-5310},
     journal={Ann. Inst. Fourier (Grenoble)},
      volume={31},
      number={1},
       pages={vii, 157\ndash 175},
         url={https://doi.org/10.5802/aif.821},
      review={\MR{613033}},
}

\bib{SteinStreet}{article}{
      author={Stein, Elias~M.},
      author={Street, Brian},
       title={Multi-parameter singular {R}adon transforms {III}: real analytic
  surfaces},
        date={2012},
     journal={Adv. Math.},
      volume={229},
      number={4},
       pages={2210\ndash 2238},
      review={\MR{2880220}},
}

\bib{Stromblac}{article}{
      author={Str\"omberg, Jan-Olov},
       title={Weak estimates on maximal functions with rectangles in certain
  directions},
        date={1977},
        ISSN={0004-2080,1871-2487},
     journal={Ark. Mat.},
      volume={15},
      number={2},
       pages={229\ndash 240},
         url={https://doi.org/10.1007/BF02386043},
      review={\MR{487260}},
}

\bib{Strequi}{article}{
      author={Str\"omberg, Jan-Olov},
       title={Maximal functions associated to rectangles with uniformly
  distributed directions},
        date={1978},
        ISSN={0003-486X},
     journal={Ann. of Math. (2)},
      volume={107},
      number={2},
       pages={399\ndash 402},
         url={https://doi.org/10.2307/1971122},
      review={\MR{481883}},
}

\bib{TaoBRR}{article}{
      author={Tao, Terence},
       title={The {B}ochner-{R}iesz conjecture implies the restriction
  conjecture},
        date={1999},
     journal={Duke Math. J.},
      volume={96},
      number={2},
       pages={363\ndash 375},
      review={\MR{1666558}},
}

\bib{WangZahl}{unpublished}{
      author={Wang, Hong},
      author={Zahl, Joshua},
       title={Volume estimates for unions of convex sets, and the {K}akeya set
  conjecture in three dimensions},
        date={2025},
        note={Preprint, \eprint{2502.17655}},
}

\bib{Wolff95}{article}{
      author={Wolff, Thomas},
       title={An improved bound for {K}akeya type maximal functions},
        date={1995},
        ISSN={0213-2230},
     journal={Rev. Mat. Iberoamericana},
      volume={11},
      number={3},
       pages={651\ndash 674},
         url={https://doi.org/10.4171/RMI/188},
      review={\MR{1363209}},
}

\bib{Zahl}{article}{
      author={Zahl, Joshua},
       title={An improved bound on the number of point-surface incidences in
  three dimensions},
        date={2013},
        ISSN={1715-0868},
     journal={Contrib. Discrete Math.},
      volume={8},
      number={1},
       pages={100\ndash 121},
      review={\MR{3118901}},
}

\bib{Zahl2}{article}{
      author={Zahl, Joshua},
       title={A {S}zemer\'edi-{T}rotter type theorem in {$\mathbb{R}^4$}},
        date={2015},
        ISSN={0179-5376},
     journal={Discrete Comput. Geom.},
      volume={54},
      number={3},
       pages={513\ndash 572},
         url={https://doi.org/10.1007/s00454-015-9717-7},
      review={\MR{3392965}},
}

\end{biblist}
\end{bibdiv}

\end{document}